\documentclass[12pt,leqno]{article}
\usepackage{amsmath, amsthm, amsfonts, amssymb, color}
\usepackage{mathrsfs}

\usepackage{enumerate}
\newcommand\dela[1]{{\color{blue}{{\small{10.2017}}}}}

\newtheorem{thm}{Theorem}[section]

\newtheorem{lem}[thm]{Lemma}

\newtheorem{Assumption}{Assumption}[section]

\theoremstyle{definition}
\newtheorem{defn}{Definition}[section]
\newcommand{\scr}[1]{\mathscr #1}
\definecolor{wco}{rgb}{0.5,0.2,0.3}

\numberwithin{equation}{section}
\newtheorem{rem}{Remark}[section]

\newcommand{\ua}{\uparrow}

\title{{\bf Small Time Asymptotics for SPDEs with Locally Monotone Coefficients}
}
\author{{\bf Shihu Li, Wei Liu\footnote{Corresponding author: weiliu@jsnu.edu.cn}, Yingchao Xie}
\\
\\
\normalsize School of Mathematics and Statistics, Jiangsu Normal University,\\ \normalsize Xuzhou 221116, China}
\date{}
\begin{document}
\maketitle
\begin{abstract} This work aims to prove the small time large deviation principle (LDP) for a class of stochastic partial differential
equations (SPDEs) with locally monotone coefficients in generalized variational framework.  The main
result could be applied to demonstrate the small time LDP for various quasilinear and semilinear SPDEs such as  stochastic porous media equations, stochastic $p$-Laplace equations, stochastic Burgers type equation, stochastic 2D Navier-Stokes equation, stochastic power law fluid equation and stochastic Ladyzhenskaya model. In particular,  our small time LDP result seems to be new in the case of general quasilinear SPDEs with multiplicative noise.
\end{abstract}
 \vskip 2cm

\def\R{\mathbb R}  \def\ff{\frac} \def\ss{\sqrt} \def\BB{\mathbb
B}
\def\N{\mathbb N} \def\kk{\kappa} \def\m{{\bf m}}
\def\dd{\delta} \def\DD{\Delta} \def\vv{\varepsilon} \def\rr{\rho}
\def\<{\langle} \def\>{\rangle} \def\GG{\Gamma} \def\gg{\gamma}
  \def\nn{\nabla} \def\pp{\partial} \def\tt{\tilde}
\def\d{\text{\rm{d}}} \def\bb{\beta} \def\aa{\alpha} \def\D{\scr D}
\def\E{\scr E} \def\si{\sigma} \def\ess{\text{\rm{ess}}}
\def\beg{\begin} \def\beq{\begin{equation}}  \def\F{\scr F}
\def\Ric{\text{\rm{Ric}}} \def\Hess{\text{\rm{Hess}}}\def\B{\mathbb B}
\def\e{\text{\rm{e}}} \def\ua{\underline a} \def\OO{\Omega}
 \def\b{\mathbf b}
\def\oo{\omega}     \def\tt{\tilde} \def\Ric{\text{\rm{Ric}}}
\def\cut{\text{\rm{cut}}} \def\P{\mathbb P} \def\ifn{I_n(f^{\bigotimes
 n})}
\def\fff{f(x_1)\dots f(x_n)} \def\ifm{I_m(g^{\bigotimes m})}
 \def\ee{\varepsilon}
\def\pm{\pi_{{\bf m}}}   \def\p{\mathbf{p}}   \def\ml{\mathbf{L}}
 \def\C{\scr C}      \def\aaa{\mathbf{r}}     \def\r{r}
\def\gap{\text{\rm{gap}}} \def\prr{\pi_{{\bf m},\varrho}}
  \def\r{\mathbf r}
\def\Z{\mathbb Z} \def\vrr{\varrho} \def\ll{\lambda}
\def\vare{{\varepsilon}}
\def \eref#1{\hbox{(\ref{#1})}}
\def\bt{\begin{theorem}}
\def\et{\end{theorem}}
\def\bl{\begin{lemma}}
\def\el{\end{lemma}}
\def\br{\begin{remark}}
\def\er{\end{remark}}
\def\bx{\begin{Example}}
\def\ex{\end{Example}}
\def\bd{\begin{definition}}
\def\ed{\end{definition}}
\def\bp{\begin{proposition}}
\def\ep{\end{proposition}}
\def\bc{\begin{corollary}}
\def\ec{\end{corollary}}

\newcommand{\ce}{\begin{eqnarray*}}
\newcommand{\de}{\end{eqnarray*}}

\section{Introduction}

The small time LDP mainly studies the asymptotic behavior of the tails of a family of probability distributions at a given point in space when the time is very small. Specifically, we focus on the limiting behavior of the solution in time interval $[0,t]$ as $t$ goes
to zero. The study of the small time asymptotics (large deviations) of finite dimensional diffusion processes was initiated
by Varadhan in the influential work \cite{Va2}.
Due to its wide applications in extremal events arising in risk management, mathematical finance, statistical mechanics, quantum physics
and many other areas, large deviation theory has become an important component of modern applied probability, see, e.g. \cite{ABBF,BBF04,BDM,C,FFK12,FJ11,FJ12,FW,St,Va1,Va2} and references therein.

Another main point being that the small time
behaviour of a diffusion process can be characterized in terms of an energy/distance
function on a Riemannian manifold, whose metric is induced from the inverse of
the diffusion coefficient, that is, such small time asymptotics will be useful to get the following Varadhan identity
\begin{equation*}
\lim_{t\rightarrow0}2t\log\mathbb{P}(X(0)\in A_1,X(t)\in A_2)=-d^2(A_1,A_2),
\end{equation*}
where $d$ is an appropriate Riemann distance associated with the diffusion generated by $X$, see, e.g. \cite{A13,AH05,CFZ,HM,Va3,Zh00} and references therein.

Apart from the above motivations, the small time asymptotic itself is also theoretically interesting, which has been studied a lot in the literatures. For instance,  the small time asymptotics of infinite dimensional diffusion processes were
studied in \cite{AK98,AZ02,FZ99,HR03,Zh00}.  Subsequently, many authors have endeavored to derive the small time LDP for different types of SPDEs. An important development concerning small time LDP for stochastic 2D Navier-Stokes equation was established by Xu and Zhang \cite{XZ09}. In \cite{RZ12}, R\"{o}ckner and Zhang studied small time LDP for stochastic 3D tamed Navier-Stokes equation. Moreover, the second named author with R\"{o}ckner and Zhu \cite{LRZ} also obtained the small time LDP for stochastic 2D quasi-geostrophic equations in the sub-critical case. The small time LDP  of stochastic 3D primitive equations was investigated by Dong and Zhang \cite{DZ18}. Recently, the small time LDP of scalar stochastic conservation laws was also studied in \cite{Zhang19}. The reader might refer to \cite{CG,CFZ,J11,LS17} and references therein for further results on this subject.

However, most papers in the literature investigated small time asymptotics (LDP) only for semilinear type SPDEs. On the other hand, some very interesting quasilinear SPDEs have been studied a
lot recently, such as stochastic porous media equation and stochastic $p$-Laplace equation, see e.g. \cite{Gess1,Gess2,BLR11,Liu09,L,LR15,LR18,MZ19,RZ2,W15}) and references therein. We would like to know whether small time asymptotics (LDP) results also hold for those SPDE models. This is one of the main motivations for us to study the small time LDP for a class of nonlinear
SPDEs, where the coefficients satisfy local monotonicity condition under the (generalized) variational framework.

The variational framework has been used intensively for studying SPDE where
the coefficients satisfying the classical monotonicity and coercivity conditions. It was first investigated in the seminal works of Pardoux \cite{Pa} and
Krylov and Rozovskii \cite{KR}, where they adapted the monotonicity tricks to prove the existence and uniqueness of solutions for a class of SPDE. Recently, this
framework has been substantially extended by  the second named author and R\"{o}ckner in \cite{LR10,LR13,LR15,LR18}
for more general class of SPDE with coefficients satisfying the generalized coercivity and local monotonicity conditions. In recent years, various  properties for SPDEs with monotone or locally monotone coefficients has been intensively investigated in the literature, such as small noise LDP \cite{L,LTZ,RZ2,XZ18}, random attractors \cite{Gess1,Gess2,BLR11,GLS}, Harnack inequality and applications \cite{Liu09}, Wong-Zakai approximation and support theorem
\cite{MZ19}, ultra-exponential
convergence \cite{W15},
and existence of optimal controls \cite{CDF}.

The proof of the main result here mainly follows the idea in Zhang's work \cite{Zh00} by using exponential equivalence arguments, which is a very powerful method  used by many
scholars to study the small time LDP for SPDEs, see, e.g. \cite{CG,DZ18,LS17,LRZ,RZ12,XZ09,Zhang19}.  More precisely, consider a zero drift stochastic differential equation with the same initial data (see \eref{Y} below), where the small noise and small time asymptotics problems are equivalent. It is easy to see that the small noise LDP for the solution $Y^\varepsilon$ of zero drift stochastic differential equation holds, thus our task is to show that the law of  $X^\varepsilon$ and $Y^{\varepsilon}$ are exponentially equivalent (see \eref{Z} below).
Comparing with some related works on small time LDP for SPDEs,
to deal with the stochastic differential equation with zero drift,
 one usually assumes that there exists another Hilbert space $H^1$ which is densely embedded in state space $H$. Working in the space of continuous $H^1$-valued trajectories, one is able to get the $H^1$-norm estimates by applying It\^{o}'s formula to $\|\cdot\|_{H^1}^2$.  However, in the variational framework, we work with the Gelfand triple $V\subset H\subset V^*$, where $V$ is a reflexive Banach space such that $V\subset H$ is continuously and densely, and it is unavailable to get the $V$-norm estimates by applying It\^{o}'s formula to $\|\cdot\|_V^2$ (e.g. in the quasilinear SPDE case). In order to overcome this difficulty, we use the concept of 2-smooth Banach space and get the $V$-norm estimates using the crucial BDG type inequality proved by Seidler \cite{Se10} (cf. \cite{ZBL} for recent generalization) for stochastic integrals in the 2-smooth Banach space, where the sharp constant $p^{1/2}$ also plays an important role in our proof. This 2-smooth Banach space is
introduced for establishing a theory of stochastic integration in Banach spaces
and typical
examples of such spaces are  $L^p$ spaces with $p\geq2$ and Sobolev spaces  $W_0^{s,p}$ with $p\geq2$ and $s\geq1$.
Thus, our
main result is applicable to various types of  SPDEs
such as stochastic porous media equation, stochastic $p$-Laplace equation,  stochastic Burgers type equation, stochastic 2D Navier-Stokes equation, stochastic power law fluid equation and stochastic Ladyzhenskaya model. In particular, by applying the abstract result to concrete models, our main result could cover the results in \cite{XZ09,LS17}, where the small time LDP for stochastic 2D Navier-Stokes equation and stochastic Ladyzhenskaya model was studied respectively. Moreover, to the best of our knowledge, our small time LDP results for general quasilinear SPDEs with multiplicative noise (such as porous media equation and $p$-Laplace equation) seem to be new in the literature.

The rest of the paper is organized as follows. In Section \ref{Sec Main Result}, we introduce the variational framework and formulate our main result.
Section \ref{Sec Proof of Thm1} is devoted to proving our main result. In Section \ref{example}, we apply the main result to various SPDE models as applications.
Throughout the paper, $C$ and $C_p$ will denote positive constants which may change from line to line, here $C_p$  emphasize the  dependence on parameter
$p$.

\section{Framework and main result} \label{Sec Main Result}
Let $(H,\langle\cdot,\cdot\rangle)$ be a real separable Hilbert space identified with
its dual space $H^*$ by the Riesz isomorphism. Let $(V,\langle\cdot,\cdot\rangle_{V})$ be a reflexive Banach space which is
continuously and densely embedded into $H$. Then we have the following Gelfand
triple:
$$V\subset H\equiv H^*\subset V^*,$$
where $V^*$
is the dual space of $V$. Let $_{V^*}\langle\cdot,\cdot\rangle_V$ denote the dualization between $V$ and $V^*$, then it follows that
$$_{V^*}\langle u,v\rangle_V=\langle u,v\rangle,~u\in H,v\in V.$$
Let $\{W(t)\}_{t\geq0}$ be a cylindrical Wiener process in a separable Hilbert space $(U,\langle\cdot,\cdot\rangle_U)$ on a complete filtered
probability space $(\Omega,\mathcal{F},\mathcal{F}_t;\mathbb{P})$. Let $(L_2(U;H),\|\cdot\|_2)$ denote the space of all Hilbert-Schmidt operators from $U$ to $H$.

In this paper, we consider the following stochastic evolution equation:
\begin{equation}\label{main equation}
\left\{ \begin{aligned}
&dX(t)=A(t,X(t))dt+B(X(t))dW(t),\\
 &X(0)=x\in H,
\end{aligned} \right.
\end{equation}
where $A:[0,T]\times V\rightarrow V^*$ and $B: V\rightarrow L_2(U;H)$ are measurable.

\medskip
Let us now state the precise conditions on the coefficients of (\ref{main equation}). 

\begin{Assumption}\label{Assumption1}
For fixed $\alpha>1$, there exist constants $\beta\geq0$, $\eta>0$, $K$ and $C$ such that the following conditions hold for all $v,v_1 ,v_2\in V$ and $t\in[0,T]$.
\begin{enumerate}
\item [$(A1)$]
(Hemicontinuity) The map $s\mapsto_{V^*}\langle A(t,v_1+sv_2),v\rangle_V$ is continuous on $\mathbb{R}$.

\item [$(A2)$]
(Local monotonicity)
 \begin{eqnarray*}
2_{V^*}\langle A(t,v_1)-A(t,v_2),v_1-v_2\rangle_V
\leq-\eta\|v_1-v_2\|_V^{\alpha}+(K+\rho(v_2))\|v_1-v_2\|_H^2,
 \end{eqnarray*}
where $\rho:V\rightarrow[0,+\infty)$ is a measurable function and locally bounded in  $V$ such that
$$\rho(v)\leq C(1+\|v\|_V^{\alpha})(1+\|v\|_H^{\beta}).$$

\item [$(A3)$] (Growth)
 \begin{eqnarray*}
 \|A(t,v)\|_{V^*}^{\frac{\alpha}{\alpha-1}}\leq C(1+\|v\|_V^{\alpha})(1+\|v\|_H^{\beta}).
 \end{eqnarray*}

\end{enumerate}
\end{Assumption}
In order to study
the small time LDP, we also need to estimate the stochastic integrals in the Banach space $V$. For a more specific example, consider the stochastic $p$-Laplace equation, it is common to take $V=W_0^{1,p}$ for $p\geq2$ as in \cite{LR15} and therefore we need to ensure
the existence of the stochastic integral in \eref{main equation} as an $W_0^{1,p}$-valued process. We recall
that the Sobolev spaces $W_0^{1,p}$
with $p\geq2$ belong to the class
of 2-smooth Banach spaces since they are isomorphic to $L^p(0,1)$ according to \cite[Remark 2 in Section 4.9]{Tr} and hence they are well suited for the stochastic It\^{o} integration (see e.g. Brz\'{e}zniak et al. \cite{Br97,BP99} for the precise construction of the stochastic integral).

In this work, we assume $V$ is a 2-smooth Banach space.
Let us denote by
$\gamma(U,V)$ the space of the $\gamma$-radonifying operators from $U$ to $V$.  We recall that $\Upsilon\in\gamma(U,V)$ if the series
$$\sum_{k\geq1}\gamma_k\Upsilon(u_k)$$
converges in $L^2(\tilde{\Omega},U)$ for any sequence $(\gamma_k)_{k\geq0}$ of independent Gaussian real-valued
random variables on a probability space $(\tilde{\Omega},\tilde{\mathcal{F}},\tilde{\mathbb{P}})$ and any orthonormal basis $(u_k)_{k\geq0}$
of $U$. Then, the space $\gamma(U,V)$ is endowed with the norm
 \begin{eqnarray*}
 \|\Upsilon\|_{\gamma(U,V)}:=\left(\tilde{\mathbb{E}}\left\|\sum_{k\geq1}\gamma_k\Upsilon(u_k)\right\|_V^2\right)^{1/2} \end{eqnarray*}
(which does not depend on $(\gamma_k)_{k\geq0}$, nor on $(u_k)_{k\geq0}$ and is a Banach space). In the following,
we shall write $\|\cdot\|_\gamma$ instead
of $\|\cdot\|_{\gamma(U,V)}$ for the simplicity of notations.
\begin{rem}\label{2-smooth}
 If $V$ is a separable Hilbert space, clearly, $V$ is 2-smooth. In this case, $\gamma(U,V)$
consists of all Hilbert-Schmidt operators of mapping $U$ into $V$, and
$ \|\cdot\|_{\gamma}=\|\cdot\|_{2}$ (see, e.g. \cite[Example 2.8]{Zh10}). Typical
examples of 2-smooth Banach space include every Hilbert space, $L^p$ spaces with $p\geq2$ and Sobolev spaces $W_0^{s,p}$ with $p\geq2$ and $s\geq1$.
\end{rem}
\medskip
We assume the following conditions on $B$.

\begin{Assumption}\label{Assumption2}
There exists constant $C$ such that the following conditions hold for all $v,v_1 ,v_2\in V$.
 \begin{eqnarray*}
&&\|B(v)\|_\gamma^2\leq C(1+\|v\|_V^2);
\\&& \|B(v_1)-B(v_2)\|_2^2\leq C\|v_1-v_2\|_H^2.
 \end{eqnarray*}
\end{Assumption}
\begin{rem}\label{c2.1}
 By Assumption \ref{Assumption1} and Assumption \ref{Assumption2}, the coercivity of $A$ and $B$ is easily obtained as
 $$_{V^*}\langle A(t,v),v\rangle_V+\|B(v)\|_2^2+\frac{\eta}{2}\|v\|_{V}^{\alpha}\leq C(1+\|v\|_{H}^2).$$
\end{rem}

\medskip
Now, we recall the following definition.

\begin{defn}\label{S.S.}
A continuous $H$-valued $(\mathcal{F}_t)$-adapted process $\{X(t)\}_{t\in[0,T]}$ is called a strong solution of \eref{main equation}, if for its $dt\otimes\mathbb{P}$-equivalence class  we have
$$X \in L^{\alpha}([0, T]\times\Omega; dt\otimes\mathbb{P};V)\cap L^2([0, T]\times\Omega; dt\otimes\mathbb{P};H)$$ and the following identity hold $\mathbb{P}$-a.s.
 \begin{eqnarray*}\label{def1}
X(t)=x+\int_0^tA(s,X(s))ds+\int_0^tB(X(s))dW(s),~~t\in[0,T].
\end{eqnarray*}
\end{defn}

The following well-posedness result is due to the second name author and R\"{o}ckner \cite[Theorem 1.1]{LR10}.

\begin{lem}\label{Lem1}
Suppose that the conditions in Assumption \ref{Assumption1} and Assumption \ref{Assumption2} hold, then \eqref{main equation} has a unique solution $\{X(t)\}_{t\in[0,T]}$ such that for any $p\geq2$
$$\mathbb{E}\left(\sup_{t\in[0,T]}\|X(t)\|_{H}^p+\int_0^T\|X(t)\|_V^{\alpha}dt\right)<\infty.$$
\end{lem}

For $\varepsilon>0$, in this paper, we aim to study the probabilistic asymptotic behavior for small time process $X(\varepsilon t)$ as $\varepsilon\rightarrow0$.

Define a functional $I(g)$ on $C([0,T];H)$ by
\begin{eqnarray}
I(g)=\inf_{h\in\Gamma_g}\left\{\frac{1}{2}\int_0^T|\dot{h}|_U^2dt\right\},\label{Ig}
\end{eqnarray}
where
\begin{eqnarray*}
\Gamma_g=\Big\{\!\!\!\!\!\!\!\!&&h\in C([0,T];H):h(\cdot) ~\text{is  absolutely continuous and such that}
\\\!\!\!\!\!\!\!\!&&g(t)=x+\int_0^tB(g(s))\dot{h}(s)ds,~t\in[0,T]\Big\}.
\end{eqnarray*}

\medskip
Now we state the main result of this paper.
\begin{thm}\label{main result}
Suppose that the conditions in Assumption \ref{Assumption1} and Assumption \ref{Assumption2}  hold. Let $\mu^{\varepsilon}$ be the law of $X(\varepsilon\cdot)$ on $C([0,T],H)$, then $\mu^{\varepsilon}$ satisfies the LDP with the rate function $I(\cdot)$ given by \eref{Ig}, i.e.,

\begin{enumerate}[(i)]
  \item  For any closed subset $F\subset C([0,T];H)$,
  $$\limsup_{\varepsilon\rightarrow0}\varepsilon\log\mu^{\varepsilon}(F)\leq-\inf_{g\in F}I(g),$$
  \item  For any open subset $G\subset C([0,T];H)$,
  $$\liminf_{\varepsilon\rightarrow0}\varepsilon\log\mu^{\varepsilon}(G)\geq-\inf_{g\in G}I(g).$$
 \end{enumerate}
\end{thm}

\begin{rem}\label{comment}
 (1) In section \ref{example} below, we will apply Theorem \ref{main result} to concrete examples
of SPDE models as applications. In particular, this covers the results in \cite{XZ09,LS17}, where the small time LDP for stochastic 2D Navier-Stokes equation and stochastic Ladyzhenskaya model was studied respectively.

 (2) Furthermore, Theorem \ref{main result} can also be applied to study the small time LDP for many other SPDEs,
such as stochastic Burgers type equation, stochastic porous media equation, stochastic $p$-Laplace equation, stochastic 2D Boussinesq equations, stochastic 2D magneto-hydrodynamic equations, stochastic 2D magnetic B\'{e}nard problem, stochastic 3D Leray-$\alpha$ model, stochastic shell models of turbulence and stochastic power law fluid equation, which seem to not have been established in the literature before.
\end{rem}

\section{Proof of main result} \label{Sec Proof of Thm1}

In this section, we will give the proof of Theorem \ref{main result}, which is mainly based on the  exponential equivalence arguments.

More precisely, for $\varepsilon>0$, by the scaling property of the Wiener process, it is easy to see that the small time process $X(\varepsilon\cdot)$ coincides in law with the solution of the following stochastic evolution equation:
\begin{eqnarray}\label{small equation}
X^{\varepsilon}(t)=x+\varepsilon\int_0^tA(\varepsilon s,X^{\varepsilon}(s))ds+\sqrt{\varepsilon}\int_0^tB(X^{\varepsilon}(s))dW(s).
\end{eqnarray}
Now, let $Y^{\varepsilon}(\cdot)$ be the solution of the following stochastic differential equation:
\begin{eqnarray}\label{Y}
Y^{\varepsilon}(t)=x+\sqrt{\varepsilon}\int_0^tB(Y^{\varepsilon}(s))dW(s)
\end{eqnarray}
and $\nu^{\varepsilon}$ be the law of $Y^{\varepsilon}(\cdot)$ on the $C([0,T];H)$. Then, applying the weak
convergence approach developed by Budhiraja and Dupuis \cite{BD00}, it easy to get that $\nu^{\varepsilon}$ satisfies the LDP with rate function $I(\cdot)$ given by \eref{Ig} (see, e.g. \cite{L,LTZ}). Therefore, our task now is to show that the two families of probability
measures $\mu^{\varepsilon}$ and $\nu^{\varepsilon}$ are exponentially equivalent, that is, for any $\delta>0$,
\begin{eqnarray}\label{Z}
\lim_{\varepsilon\rightarrow0}\varepsilon\log\mathbb{P}\left(\sup_{0\leq t\leq T}\|X^{\varepsilon}(t)-Y^{\varepsilon}(t)\|_{H}^2>\delta\right)=-\infty.
\end{eqnarray}
Then Theorem \ref{main result} follows from the fact that if one of the two exponentially equivalent
families satisfies the LDP, so does the other (see e.g. \cite[Theorem 4.2.13]{DZ}).

We begin the proof with the following lemma which provides an estimate of the probability that the solution
of \eref{small equation} leaves an energy ball.

 Set
$$\left(|X^\varepsilon|_{H,V}(T)\right)^p:=\sup_{0\leq t\leq T}\|X^{\varepsilon}(t)\|_H^p+\varepsilon\int_0^T\|X^{\varepsilon}(t)\|_H^{p-2}
\|X^{\varepsilon}(t)\|_V^{\alpha}dt.$$
Then, we claim that

\begin{lem} \label{3.1} For any $p\geq2$,
\begin{align}
\lim_{M\rightarrow\infty}\sup_{0<\varepsilon\leq1}\varepsilon\log\mathbb{P}\left(\left(|X^\varepsilon|_{H,V}(T)\right)^p>M\right)=-\infty.
\end{align}
\end{lem}
\begin{proof}
According to It\^{o}'s formula (cf. \cite[Theorem 4.2.5]{LR15}) and Remark \ref{c2.1}, we have
\begin{eqnarray*}\label{ItoFormu 000}
\|X^{\varepsilon}(t)\|_H^{p}=\!\!\!\!\!\!\!\!&&\|x\|_H^{p}+\frac{p(p-2)}{2}\varepsilon\int_{0} ^{t}\|X^{\vare}(s)\|_H^{p-4}\|B(X^{\vare}(s))^*X^{\vare}(s) \|_{U}^2ds \nonumber \\
 \!\!\!\!\!\!\!\!&& + \frac{p}{2}\varepsilon\int_{0} ^{t}\|X^{\vare}(s)\|_H^{p-2}(2_{V^*}\langle A(\vare s,X^{\vare}(s)),X^{\vare}(s)\rangle_V+\|B(X^{\vare}(s))\|_2^2)ds
 \\ \!\!\!\!\!\!\!\!&& +p\sqrt{\varepsilon}\int_{0} ^{t}\|X^{\vare}(s)\|_H^{p-2}\langle X^{\vare}(s),B(X^{\vare}(s))dW(s)\rangle
 \\\leq\!\!\!\!\!\!\!\!&&\|x\|_H^{p}-\frac{p\eta}{4}\varepsilon\int_{0} ^{t}\|X^{\vare}(s)\|_H^{p-2}\|X^{\vare}(s) \|_{V}^{\alpha}ds
 +C\varepsilon\int_0^t\|X^{\vare}(s)\|_H^{p}ds
 \\ \!\!\!\!\!\!\!\!&& +p\sqrt{\varepsilon}\int_{0} ^{t}\|X^{\vare}(s)\|_H^{p-2}\langle X^{\vare}(s),B(X^{\vare}(s))dW(s)\rangle.
\end{eqnarray*}
Then it is easy to get
\begin{eqnarray*}
\|X^{\varepsilon}(t)\|_H^{p}+\varepsilon\int_{0} ^{t}\|X^{\vare}(s)\|_H^{p-2}\|X^{\vare}(s)\|_{V}^{\alpha}ds
\leq\!\!\!\!\!\!\!\!&& C_p\left(1+\|x\|_H^{p}\right)+C_p\varepsilon\int_0^t\|X^{\vare}(s)\|_H^{p}ds
 \\\!\!\!\!\!\!\!\!&& +C_p\sqrt{\varepsilon}\int_{0} ^{t}\|X^{\vare}(s)\|_H^{p-2}\langle X^{\vare}(s),B(X^{\vare}(s))dW(s)\rangle,
\end{eqnarray*}
which implies that for any $q\geq2$
\begin{eqnarray*}
\left(\mathbb{E}\left(|X^\varepsilon|_{H,V}(T)\right)^{p})^q\right)^{1/q}
\leq\!\!\!\!\!\!\!\!&& C_p\left(1+\|x\|_H^{p}\right)+C_p\varepsilon\left(\mathbb{E}\left(\int_0^T\left(|X^\varepsilon|_{H,V}(t)\right)^{p}dt\right)^q\right)^{1/q}
 \\\!\!\!\!\!\!\!\!&& +C_p\sqrt{\varepsilon}\left(\mathbb{E}\sup_{0\leq t\leq T}\left|\int_{0} ^{t}\|X^{\vare}(s)\|_H^{p-2}\langle X^{\vare}(s),B(X^{\vare}(s))dW(s)\rangle\right|^q\right)^{1/q}.
\end{eqnarray*}
To estimate the stochastic integral term, we will use the following martingale inequality  from \cite{Da}
that there exists a universal constant $C$ such that, for any $q\geq2$ and for any continuous martingale $M_t$
with $M_0=0$, one has
\begin{eqnarray}
(\mathbb{E}(|M_t^*|^q))^{1/q}\leq Cq^{1/2}(\mathbb{E}\langle M\rangle_t^{q/2})^{1/q},\label{martingale inequality}
\end{eqnarray}
where $M_t^*=\sup_{0\leq s\leq t}|M_s|$.

Hence, we can get
\begin{eqnarray*}
&&\sqrt{\varepsilon}\left(\mathbb{E}\sup_{0\leq t\leq T}\left|\int_{0} ^{t}\|X^{\vare}(s)\|_H^{p-2}\langle X^{\vare}(s),B(X^{\vare}(s))dW(s)\rangle\right|^{q}\right)^{1/q}
\\\leq\!\!\!\!\!\!\!\!&&C\sqrt{q\varepsilon}\left(\mathbb{E}\left(\int_{0} ^{T}\|X^{\vare}(t)\|_H^{2p-2}\|B(X^{\vare}(t))\|_2^2dt\right)^{q/2}\right)^{1/q}\\
\leq\!\!\!\!\!\!\!\!&&C\sqrt{q\varepsilon}\left(\mathbb{E}\left(\int_{0} ^{T}\|X^{\vare}(t)\|_H^{2p-2}(1+\|X^{\vare}(t)\|_H^2)dt\right)^{q/2}\right)^{1/q}
\\\leq\!\!\!\!\!\!\!\!&&C\sqrt{q\varepsilon}\left(\int_{0} ^{T}1+\left(\mathbb{E}\|X^{\vare}(t)\|_H^{pq}\right)^{2/q}dt\right)^{1/2}.
\end{eqnarray*}
Therefore, combining the above estimates yields
\begin{eqnarray*}
\left(\mathbb{E}\left(|X^\varepsilon|_{H,V}(T)\right)^{p})^q\right)^{2/q}
\leq\!\!\!\!\!\!\!\!&& C_p\left(1+\|x\|_H^{p}\right)^2+C_p\varepsilon^2\left(\mathbb{E}\left(\int_0^T\left(|X^\varepsilon|_{H,V}(t)\right)^{p}dt\right)^q\right)^{2/q}
 \\\!\!\!\!\!\!\!\!&& +C_pq\varepsilon\left(\int_{0} ^{T}1+\left(\mathbb{E}\|X^{\vare}(t)\|_H^{pq}\right)^{2/q}dt\right)
\\ \leq\!\!\!\!\!\!\!\!&& C_p\left(1+\|x\|_H^{p}\right)^2+C_p\varepsilon^2\left(\int_0^T\left(\mathbb{E}(|X^\varepsilon|_{H,V}(t))^{pq}\right)^{2/q}dt\right)
 \\\!\!\!\!\!\!\!\!&& +C_pq\varepsilon T+C_pq\varepsilon\left(\int_0^T\left(\mathbb{E}(|X^\varepsilon|_{H,V}(t))^{pq}\right)^{2/q}dt\right),
\end{eqnarray*}
where the second inequality is due to Minkowski's inequality.

Applying Gronwall's lemma we obtain that
\begin{eqnarray*}
\left(\mathbb{E}\left(|X^\varepsilon|_{H,V}(T)\right)^{p})^q\right)^{2/q}
\leq\!\!\!\!\!\!\!\!&&\left( C_p\left(1+\|x\|_H^{p}\right)^2+C_pq\varepsilon T\right)\cdot\exp{(C_p\varepsilon^2+C_pq\varepsilon)}.
\end{eqnarray*}

Using Chebyshev's inequality, for any $M>0$, we have
\begin{eqnarray*}
\mathbb{P}\left(\left(|X^\varepsilon|_{H,V}(T)\right)^p>M\right)\leq \frac{\mathbb{E}\left(|X^\varepsilon|_{H,V}(T)\right)^{p})^q}{M^{q}}.
\end{eqnarray*}

Taking $q=2/{\varepsilon}$, we get
\begin{eqnarray*}
&&\varepsilon\log\mathbb{P}\left(\left(|X^\varepsilon|_{H,V}(T)\right)^p>M\right)
\nonumber\\\leq\!\!\!\!\!\!\!\!&&-2\log M+\log\left(\mathbb{E}\left(|X^\varepsilon|_{H,V}(T)\right)^{p})^q\right)^{2/q}\nonumber\\
\leq\!\!\!\!\!\!\!\!&&-2\log M+C_p\varepsilon^2+2C_p+\log\left( C_p\left(1+\|x\|_H^{p}\right)^2+2C_pT\right).
\end{eqnarray*}
Then, it is easy to see
\begin{eqnarray}\label{2.2}
\sup_{0<\varepsilon\leq1}\varepsilon\log\mathbb{P}\left(\left(|X^\varepsilon|_{H,V}(T)\right)^p>M\right)
\leq\!\!\!\!\!\!\!\!&&-2\log M+3C_p
\nonumber\\\!\!\!\!\!\!\!\!&&+\log\left( C_p\left(1+\|x\|_H^{p}\right)^2+2C_p T\right).
\end{eqnarray}
Let $M\rightarrow\infty$ on both sides of \eref{2.2}, we complete the proof.
\end{proof}

Since $V$ is dense in $H$, there exists a sequence $\{x_n\}\subset V$ such that
$$\lim_{n\rightarrow+\infty}\|x_n-x\|_H=0.$$
Let $X_n^{\varepsilon}$
be the solution of \eref{main equation} with the initial value $x_n$. From the proof of Lemma \ref{3.1}, it follows that
\begin{align}\label{2.3}
\lim_{M\rightarrow\infty}\sup_n\sup_{0<\varepsilon\leq1}\varepsilon\log\mathbb{P}\left(\left(|X_n^\varepsilon|_{H,V}(T)\right)^p>M\right)=-\infty.
\end{align}

Let $Y_n^{\varepsilon}$ be the solution of \eref{Y} with the initial value $x_n$, i.e.
\begin{eqnarray}
Y_n^{\varepsilon}(t)=x_n+\sqrt{\varepsilon}\int_0^tB(Y_n^{\varepsilon}(s))dW(s).\label{Yn}
\end{eqnarray}
Then we can get the following estimate.
\begin{lem} \label{3.2} For any $n\in\mathbb{Z}^+$,
\begin{align*}
\lim_{M\rightarrow\infty}\sup_{0<\varepsilon\leq1}\varepsilon\log\mathbb{P}\left(\sup_{0\leq t\leq T}\|Y_n^\varepsilon(t)\|_{V}^2>M\right)=-\infty.
\end{align*}
\end{lem}

\begin{proof}

To estimate the stochastic integral term in the Banach space $V$, here we need to use the BDG type inequality for 2-smooth Banach space (cf. \cite[Theorem 1.1]{Se10}). Then
for any $q>1$, we can get
\begin{eqnarray*}
\left(\mathbb{E}\left(\sup_{0\leq t\leq T}\|Y^\varepsilon_n(t)\|_V^{2q}\right)\right)^{\frac{1}{2q}}\leq\!\!\!\!\!\!\!\!&&\|x_n\|_V+\sqrt{\varepsilon}\mathbb{E}\left(\sup_{0\leq t\leq T}\left\|\int_0^tB(Y_n^{\varepsilon}(s))dW(s)\right\|_V^{2q}\right)^{\frac{1}{2q}}
\\\leq\!\!\!\!\!\!\!\!&&\|x_n\|_V+\sqrt{2q\varepsilon}C\left(\mathbb{E}\left(\int_0^T\left\|B(Y_n^{\varepsilon}(s)\right\|_\gamma^{2}ds\right)^{q}\right)^{\frac{1}{2q}}
\\\leq\!\!\!\!\!\!\!\!&&\|x_n\|_V+\sqrt{2q\varepsilon}C\left(\mathbb{E}\left(\int_0^T(1+\left\|Y_n^{\varepsilon}(s)\right\|_V^{2})ds\right)^{q}\right)^{\frac{1}{2q}}
\\\leq\!\!\!\!\!\!\!\!&&\|x_n\|_V+\sqrt{2q\varepsilon}C\left(\int_0^T(1+\left(\mathbb{E}\|Y_n^{\varepsilon}(s)\|_V^{2q})\right)^{1/q}ds\right)^{\frac{1}{2}},
\end{eqnarray*}
where in the last inequality we use Minkowski's inequality and the constant $C$ is independent of $q$ and $\varepsilon$.

Then, it is easy to get
\begin{eqnarray*}
\left(\mathbb{E}\left(\sup_{0\leq t\leq T}\|Y^\varepsilon_n(t)\|_V^{2q}\right)\right)^{\frac{1}{q}}\leq2\|x_n\|_V^2+2q\varepsilon C+2q\varepsilon C\int_0^T\left(\mathbb{E}\|Y_n^{\varepsilon}(s)\|_V^{2q}\right)^{1/q}ds.
\end{eqnarray*}
Applying Gronwall's Lemma yields
\begin{eqnarray*}
\left(\mathbb{E}\left(\sup_{0\leq t\leq T}\|Y^\varepsilon_n(t)\|_V^{2q}\right)\right)^{\frac{1}{q}}\leq\left(2\|x_n\|_V^2+2q\varepsilon C\right)e^{2q\varepsilon C}.
\end{eqnarray*}
Fixing $M$ and taking $q=1/\varepsilon$, we have
\begin{eqnarray}\label{2.2v}
\varepsilon\log\mathbb{P}\left(\sup_{0\leq t\leq T}\|Y_n^\varepsilon(t)\|_{V}^2>M\right)\leq\!\!\!\!\!\!\!\!&&\varepsilon\log\frac{\mathbb{E}\left(\sup_{0\leq t\leq T}\|Y^\varepsilon_n(t)\|_V^{2q}\right)}{M^q}
\nonumber\\\leq\!\!\!\!\!\!\!\!&&-\log M+\log\left(\mathbb{E}\left(\sup_{0\leq t\leq T}\|Y_n^\varepsilon(t)\|_{V}^{2q}\right)\right)^{1/q}\nonumber\\
\leq\!\!\!\!\!\!\!\!&&-\log M+\log\left(2\|x_n\|_V^2+2C\right)+2C.
\end{eqnarray}
Let $M\rightarrow\infty$ on both sides of \eref{2.2v}, we complete the proof.
\end{proof}

Now, we establish the exponential convergence of $X_n^{\varepsilon}-X^\varepsilon$.

\begin{lem} \label{3.3}
For any $\delta>0$, we have
\begin{align*}
\lim_{n\rightarrow\infty}\sup_{0<\varepsilon\leq1}\varepsilon\log\mathbb{P}\left(\sup_{0\leq t\leq T}\|X_n^{\varepsilon}(t)-X^\varepsilon(t)\|_{H}^2>\delta\right)=-\infty.
\end{align*}
\end{lem}

\begin{proof}

For $M>0$, define stopping time
$$\tau_{M}^{\varepsilon}=\inf\left\{t:\varepsilon\int_0^t\|X^{\varepsilon}(s)\|_{H}^{p-2}
\|X^{\varepsilon}(s)\|_{V}^{\alpha}ds>M, \text{or}~ \|X^{\varepsilon}(t)\|_{H}^2>M\right\}.$$
Using It\^{o}'s formula, the local monotone condition (A2), we can get
\begin{eqnarray}\label{2.4}
&&e^{-\varepsilon\int_0^{t\wedge\tau_{M}^{\varepsilon}}(K+\rho(X^{\varepsilon}(s)))ds}\|X^{\varepsilon}(t)-X_n^{\varepsilon}(t)\|_{H}^2
\nonumber\\=\!\!\!\!\!\!\!\!&&
\|x-x_n\|_H^2-\varepsilon\int_0^{t\wedge\tau_{M}^{\varepsilon}}e^{-\varepsilon\int_0^{s\wedge\tau_{M}^{\varepsilon}}(K+\rho(X^{\varepsilon}(r)))dr}(K+\rho(X^{\varepsilon}(s)))\|X^{\varepsilon}(s)-X_n^{\varepsilon}(s)\|_{H}^2ds
\nonumber\\\!\!\!\!\!\!\!\!&&+\varepsilon\int_0^{t\wedge\tau_{M}^{\varepsilon}}
e^{-\varepsilon\int_0^{s\wedge\tau_{M}^{\varepsilon}}(K+\rho(X^{\varepsilon}(r)))dr}\Big (2_{V^*}\langle A(X^{\varepsilon}(s))-A(X_n^{\varepsilon}(s)),X^{\varepsilon}(s)-X_n^{\varepsilon}(s)\rangle_V \nonumber\\\!\!\!\!\!\!\!\!&&+\|B(X^{\varepsilon}(s))-B(X_n^{\varepsilon}(s))\|_2^2\Big)ds
\nonumber\\\!\!\!\!\!\!\!\!&&+2\sqrt{\varepsilon}\int_0^{t\wedge\tau_{M}^{\varepsilon}}
e^{-\varepsilon\int_0^{s\wedge\tau_{M}^{\varepsilon}}(K+\rho(X^{\varepsilon}(r)))dr}
\langle X^{\varepsilon}(s)-X_n^{\varepsilon}(s),(B(X^{\varepsilon}(s))-B(X_n^{\varepsilon}(s)))dW(s)\rangle
\nonumber\\\leq\!\!\!\!\!\!\!\!&&
\|x-x_n\|_H^2
\nonumber\\\!\!\!\!\!\!\!\!&&+2\sqrt{\varepsilon}\int_0^{t\wedge\tau_{M}^{\varepsilon}}
e^{-\varepsilon\int_0^{s\wedge\tau_{M}^{\varepsilon}}(K+\rho(X^{\varepsilon}(r)))dr}
\langle X^{\varepsilon}(s)-X_n^{\varepsilon}(s),(B(X^{\varepsilon}(s))-B(X_n^{\varepsilon}(s)))dW(s)\rangle.
\nonumber\end{eqnarray}
Then by the martingale inequality \eref{martingale inequality}, we obtain
\begin{eqnarray*}
&&\left(\mathbb{E}\left[\sup_{0\leq s\leq t\wedge\tau_{M}^{\varepsilon}}e^{-\varepsilon\int_0^{s}(K+\rho(X^{\varepsilon}(r)))dr}
\|X^{\varepsilon}(s)-X_n^{\varepsilon}(s)\|_{H}^2\right]^q\right)^{2/q}
\\\leq\!\!\!\!\!\!\!\!&&
2\|x-x_n\|_H^4+C(q\varepsilon+t\varepsilon^2)\int_0^{t}
\Big(\mathbb{E}\Big[\sup_{0\leq r\leq s\wedge\tau_{M}^{\varepsilon}}e^{-\varepsilon\int_0^{s}(K+\rho(X^{\varepsilon}(r)))dr}
\|X^{\varepsilon}(s)-X_n^{\varepsilon}(s)\|_{H}^2\Big]^q\Big)^{2/q}ds.
\end{eqnarray*}
Applying Gronwall's lemma yields
\begin{eqnarray*}
&&\left(\mathbb{E}\left[\sup_{0\leq s\leq T\wedge\tau_{M}^{\varepsilon}}e^{-\varepsilon\int_0^{s}(K+\rho(X^{\varepsilon}(r)))dr}
\|X^{\varepsilon}(s)-X_n^{\varepsilon}(s)\|_{H}^2\right]^q\right)^{2/q}
\leq
2\|x-x_n\|_H^4e^{CT(q\varepsilon+T\varepsilon^2)}.
\end{eqnarray*}
Hence, we have
\begin{eqnarray*}
&&\left(\mathbb{E}\left[\sup_{0\leq s\leq T\wedge\tau_{M}^{\varepsilon}}
\|X^{\varepsilon}(s)-X_n^{\varepsilon}(s)\|_{H}^2\right]^q\right)^{2/q}
\\\leq\!\!\!\!\!\!\!\!&&\left(\mathbb{E}\left[\sup_{0\leq s\leq T\wedge\tau_{M}^{\varepsilon}}e^{-\varepsilon\int_0^{s}(K+\rho(X^{\varepsilon}(r)))dr}
\|X^{\varepsilon}(s)-X_n^{\varepsilon}(s)\|_{H}^2e^{q\varepsilon\int_0^{s}(K+\rho(X^{\varepsilon}(r)))dr}\right]^q\right)^{2/q}
\\\leq\!\!\!\!\!\!\!\!&&2\|x-x_n\|_H^4e^{CT(M+K+q\varepsilon+T\varepsilon^2)}.
\end{eqnarray*}
Fixing $M$ and taking $q=2/{\varepsilon}$ we get
\begin{eqnarray}\label{2.5}
&&\sup_{0<\varepsilon\leq1}\varepsilon\log\mathbb{P}\left(\sup_{0\leq t\leq T\wedge\tau_{M}^{\varepsilon}}\|X^{\varepsilon}(t)-X_n^{\varepsilon}(t)\|_{H}^2>\delta\right)
\nonumber\\\leq\!\!\!\!\!\!\!\!&&
\sup_{0<\varepsilon\leq1}\varepsilon\log\frac{\mathbb{E}\left[\sup_{0\leq t\leq T\wedge\tau_{M}^{\varepsilon}}\|X^{\varepsilon}(t)-X_n^{\varepsilon}(t)\|_{H}^{2q}\right]}{\delta^q}
\nonumber\\
\leq\!\!\!\!\!\!\!\!&&\log 2\|x-x_n\|_H^4+C_{T,M,K}-2\log\delta
\nonumber\\
\rightarrow\!\!\!\!\!\!\!\!&&-\infty,~~\text{as}~n\rightarrow+\infty.
\end{eqnarray}

By Lemma \ref{3.1}, for any $R>0$, there exists a constant $M$ such that for every $\varepsilon\in(0,1]$ the following inequality holds:
\begin{eqnarray}\label{2.6}
\mathbb{P}\left(\left(|X^\varepsilon|_{H,V}(T)\right)^p>M\right)\leq e^{-R/\varepsilon}.
\end{eqnarray}
For such $M$, by \eref{2.5} and the definition of stoping time $\tau_{M}^{\varepsilon}$, there exists a positive integer $N$, such that for any $n\geq N$,
\begin{eqnarray}\label{2.7}
&&\sup_{0<\varepsilon\leq1}\varepsilon\log\mathbb{P}\left(\sup_{0\leq t\leq T}\|X^{\varepsilon}(t)-X_n^{\varepsilon}(t)\|_{H}^2>\delta,\left(|X^\varepsilon|_{H,V}(T)\right)^p\leq M\right)
\nonumber\\\leq\!\!\!\!\!\!\!\!&&
\sup_{0<\varepsilon\leq1}\varepsilon\log\mathbb{P}\left(\sup_{0\leq t\leq T\wedge\tau_{M}^{\varepsilon}}\|X^{\varepsilon}(t)-X_n^{\varepsilon}(t)\|_{H}^2>\delta\right)\leq-R.
\end{eqnarray}
Combining \eref{2.5} and \eref{2.7}, we conclude that there exists a positive integer $N$ such that for any $n\geq N$
and $\varepsilon\in(0,1]$,
\begin{eqnarray*}
\mathbb{P}\left(\sup_{0\leq t\leq T}\|X^{\varepsilon}(t)-X_n^{\varepsilon}(t)\|_{H}^2>\delta\right)\leq2e^{-R/\varepsilon}.
\end{eqnarray*}
Since $R$ is arbitrary, the assertion of the lemma follows.
\end{proof}

We also need to establish the exponential convergence of $Y_n^{\varepsilon}-Y^\varepsilon$.

\begin{lem} \label{3.4}
For any $\delta>0$, we have
\begin{align}\label{2.8}
\lim_{n\rightarrow\infty}\sup_{0<\varepsilon\leq1}\varepsilon\log\mathbb{P}\left(\sup_{0\leq t\leq T}\|Y_n^{\varepsilon}(t)-Y^\varepsilon(t)\|_{H}^2>\delta\right)=-\infty.
\end{align}
\end{lem}

\begin{proof}
From \eref{Y} and \eref{Yn}, it is easy to see
\begin{eqnarray*}
Y^{\varepsilon}(t)-Y_n^{\varepsilon}(t)=x-x_n+\sqrt{\varepsilon}\int_0^t(B(Y_n^{\varepsilon}(s))-B(Y^{\varepsilon}(s)))dW(s).
\end{eqnarray*}
Applying It\^{o}'s formula to $\|Y_n^{\varepsilon}(t)-Y^\varepsilon(t)\|_{H}^2$, we have
\begin{eqnarray*}
&&\|Y^{\varepsilon}(t)-Y_n^{\varepsilon}(t)\|_{H}^2
\nonumber\\=\!\!\!\!\!\!\!\!&&
\|x-x_n\|_H^2+\varepsilon\int_0^{t}\|B(Y^{\varepsilon}(s))-B(Y_n^{\varepsilon}(s))\|_2^2ds
\nonumber\\\!\!\!\!\!\!\!\!&&+2\sqrt{\varepsilon}\int_0^{t}
\langle Y^{\varepsilon}(s)-Y_n^{\varepsilon}(s),(B(Y^{\varepsilon}(s))-B(Y_n^{\varepsilon}(s)))dW(s)\rangle.
\nonumber\end{eqnarray*}
Then by the Assumption \eref{Assumption2} and martingale inequality \eref{martingale inequality}, we obtain
\begin{eqnarray*}
&&\left(\mathbb{E}\left(\|Y^{\varepsilon}(t)-Y_n^{\varepsilon}(t)\|_{H}^{2q}\right)\right)^{\frac{2}{q}}
\nonumber\\\leq\!\!\!\!\!\!\!\!&&
2\|x-x_n\|_H^4+(\varepsilon^2C+\varepsilon Cq)\int_0^{t}\left(\mathbb{E}\left(\|Y^{\varepsilon}(s)-Y_n^{\varepsilon}(s)\|_{H}^{2q}\right)\right)^{\frac{2}{q}}ds,
\nonumber\end{eqnarray*}
where the constant $C$ is independent of $q$ and $\varepsilon$.

Utilizing Gronwall's lemma, we get
\begin{eqnarray*}
&&\left(\mathbb{E}\left(\|Y^{\varepsilon}(t)-Y_n^{\varepsilon}(t)\|_{H}^{2q}\right)\right)^{\frac{2}{q}}
\leq
2\|x-x_n\|_H^4+\exp(\varepsilon^2C+\varepsilon Cq).\end{eqnarray*}
Applying the same argument as the proof of \eref{2.5} in Lemma \ref{3.3}, we complete the proof.
\end{proof}

The following lemma says that for a fixed integer $n$, the two families $\{X_n^{\varepsilon},{\varepsilon}>0\}$ and $\{Y_n^{\varepsilon},{\varepsilon}>0\}$
are exponentially equivalent.

\begin{lem} \label{3.5}
For any $\delta>0$ and any positive integer $n$, we have
$$\lim_{\varepsilon\rightarrow0}\varepsilon\log\mathbb{P}\left(\sup_{0\leq t\leq T}\|X^{\varepsilon}_n(t)-Y_n^{\varepsilon}(t)\|_{H}^2>\delta\right)=-\infty.$$
\end{lem}

\begin{proof}
For $M>0$, we define the following stopping times
$$\tau_{M,\varepsilon}^{1,n}=\inf\left\{t:\varepsilon\int_0^t\|X_n^{\varepsilon}(s)\|_{H}^{p-2}
\|X_n^{\varepsilon}(s)\|_{V}^{\alpha}ds>M, \text{or}~ \|X_n^{\varepsilon}(t)\|_{H}^2>M\right\},$$
$$\tau_{M,\varepsilon}^{2,n}=\inf\left\{t: \|Y_n^{\varepsilon}(t)\|_{V}^2>M\right\}.$$
Setting $\tau_{M,\varepsilon}^{n}=\tau_{M,\varepsilon}^{1,n}\wedge\tau_{M,\varepsilon}^{2,n}$, then we can get by
applying It\^{o}'s formula that
\begin{eqnarray*}
&&\|X_n^{\varepsilon}(t\wedge\tau_{M,\varepsilon}^{n})-Y_n^{\varepsilon}(t\wedge\tau_{M,\varepsilon}^{n})\|_H^{2}
\\=\!\!\!\!\!\!\!\!&&\varepsilon\int_{0} ^{t\wedge\tau_{M,\varepsilon}^{n}}(2_{V^*}\langle A(\vare s,X_n^{\vare}(s))-A(\vare s,Y_n^{\vare}(s),X_n^{\vare}(s)-Y_n^{\vare}(s)\rangle_V+\|B(X_n^{\vare}(s))-B(Y_n^{\vare}(s))\|_2^2)ds
 \\ \!\!\!\!\!\!\!\!&& +\varepsilon\int_{0} ^{t\wedge\tau_{M,\varepsilon}^{n}}2_{V^*}\langle A(\vare s,Y_n^{\vare}(s)),X_n^{\vare}(s)-Y_n^{\vare}(s)\rangle_Vds
  \\ \!\!\!\!\!\!\!\!&& +2\sqrt{\varepsilon}\int_{0} ^{t\wedge\tau_{M,\varepsilon}^{n}}\langle X_n^{\vare}(s)-X_n^{\vare}(s),(B(X_n^{\vare}(s))-B(Y_n^{\vare}(s)))dW(s)\rangle.
\end{eqnarray*}
By condition (A3) and Young's inequality, we have
\begin{eqnarray*}
2_{V^*}\langle A(\vare s,Y_n^{\vare}(s)),X_n^{\vare}(s)-Y_n^{\vare}(s)\rangle_V
\leq\!\!\!\!\!\!\!\!&& 2\|A(\vare s,Y_n^{\vare}(s))\|_{V^*}\|X_n^{\vare}(s)-Y_n^{\vare}(s)\|_V
\\ \leq\!\!\!\!\!\!\!\!&& C\|A(\vare s,Y_n^{\vare}(s))\|_{V^*}^{\frac{\alpha-1}{\alpha}}+\theta\|X_n^{\vare}(s)-Y_n^{\vare}(s)\|_V^\alpha
\\ \leq\!\!\!\!\!\!\!\!&& C(1+\|Y_n^{\vare}(s)\|_V^{\alpha+\beta})+\theta\|X_n^{\vare}(s)-Y_n^{\vare}(s)\|_V^\alpha,
\end{eqnarray*}
where $\theta<\eta$ is a small positive constant. Then by condition (A2), we obtain that
\begin{eqnarray*}
&&\|X_n^{\varepsilon}(t\wedge\tau_{M,\varepsilon}^{n})-Y_n^{\varepsilon}(t\wedge\tau_{M,\varepsilon}^{n})\|_H^{2}
\\\leq\!\!\!\!\!\!\!\!&&\varepsilon\int_{0} ^{t\wedge\tau_{M,\varepsilon}^{n}}(K+\rho(X_n^{\vare}(s))\|X_n^{\vare}(s)-Y_n^{\vare}(s)\|_H^2ds
 \\ \!\!\!\!\!\!\!\!&& +\varepsilon\int_{0} ^{t\wedge\tau_{M,\varepsilon}^{n}}C(1+\|Y_n^{\vare}(s)\|_V^{\alpha+\beta})-(\eta-\theta)\|X_n^{\vare}(s)-Y_n^{\vare}(s)\|_V^\alpha ds
  \\ \!\!\!\!\!\!\!\!&& +2\sqrt{\varepsilon}\int_{0} ^{t\wedge\tau_{M,\varepsilon}^{n}}\langle X_n^{\vare}(s)-X_n^{\vare}(s),(B(X_n^{\vare}(s))-B(Y_n^{\vare}(s)))dW(s)\rangle.
\end{eqnarray*}
Then applying Gronwall's lemma yields that
\begin{eqnarray*}
&&\|X_n^{\varepsilon}(t\wedge\tau_{M,\varepsilon}^{n})-Y_n^{\varepsilon}(t\wedge\tau_{M,\varepsilon}^{n})\|_H^{2}
\\\leq\!\!\!\!\!\!\!\!&&\Big[C\varepsilon\int_{0} ^{t\wedge\tau_{M,\varepsilon}^{n}}(1+\|Y_n^{\vare}(s)\|_V^{\alpha+\beta}) ds
  \\ \!\!\!\!\!\!\!\!&& +2\sqrt{\varepsilon}\int_{0} ^{t\wedge\tau_{M,\varepsilon}^{n}}\langle X_n^{\vare}(s)-X_n^{\vare}(s),(B(X_n^{\vare}(s))-B(Y_n^{\vare}(s)))dW(s)\rangle\Big]
\\ \!\!\!\!\!\!\!\!&&\cdot e^{\varepsilon\int_0^{t\wedge\tau_{M,\varepsilon}}(K+\rho(X_n^{\vare}(s))ds}.
\end{eqnarray*}
Using \eref{martingale inequality}, we obtain by the definition of the stopping time
$\tau_{M,\varepsilon}^{n}$ that
\begin{eqnarray*}
&&\left(\mathbb{E}\left[\sup_{0\leq s\leq t\wedge\tau_{M,\varepsilon}^{n}}\|X_n^{\varepsilon}(s)-Y_n^{\varepsilon}(s)\|_H^{2}\right]^q\right)^{2/q}
\\\leq\!\!\!\!\!\!\!\!&&Ce^{M+(M+C)t\varepsilon}\Big[C(t\varepsilon+M^{\frac{\alpha+\beta}{2}}t\varepsilon)^2+q\varepsilon
  \\ \!\!\!\!\!\!\!\!&& \cdot \int_0^t\left(\mathbb{E}\left[\sup_{0\leq r\leq s\wedge\tau_{M,\varepsilon}^{n}}\|X_n^{\varepsilon}(r)-Y_n^{\varepsilon}(r)\|_H^{2}\right]^q\right)^{2/q}ds\Big].
\end{eqnarray*}
Applying Gronwall's lemma again, we obtain that
\begin{eqnarray*}
&&\left(\mathbb{E}\left[\sup_{0\leq s\leq T\wedge\tau_{M,\varepsilon}^{n}}\|X_n^{\varepsilon}(s)-Y_n^{\varepsilon}(s)\|_H^{2}\right]^q\right)^{2/q}
\\\leq\!\!\!\!\!\!\!\!&&Ce^{M+(M+C)T\varepsilon}(T\varepsilon+M^{\frac{\alpha+\beta}{2}}T\varepsilon)^2\exp\left[CqT\varepsilon e^{M+(M+C)T\varepsilon}\right].
\end{eqnarray*}
Fixing $M$ and taking $q=2/\varepsilon$ we have
\begin{eqnarray}\label{2.9}
&&\varepsilon\log\mathbb{P}\left(\sup_{0\leq t\leq T\wedge\tau_{M,\varepsilon}^{n}}\|X_n^{\varepsilon}(t)-Y_n^{\varepsilon}(t)\|_{H}^2>\delta\right)
\nonumber\\\leq\!\!\!\!\!\!\!\!&&
\varepsilon\log\frac{\mathbb{E}\left[\sup_{0\leq t\leq T\wedge\tau_{M,\varepsilon}^{n}}\|X_n^{\varepsilon}(t)-Y_n^{\varepsilon}(t)\|_{H}^{2q}\right]}{\delta^q}
\nonumber\\
\leq\!\!\!\!\!\!\!\!&&\log C(T\varepsilon+M^{\frac{\alpha+\beta}{2}}T\varepsilon)^2 +CTe^{M+(M+C)T\varepsilon}+{M+(M+C)T\varepsilon}-2\log\delta
\nonumber\\
\rightarrow\!\!\!\!\!\!\!\!&&-\infty,~~\text{as}~\varepsilon\rightarrow0.
\end{eqnarray}
By \eref{2.3} and Lemma \ref{3.2}, for any $R>0$, there exists a constant $M$ such that the following inequalities hold:
\begin{eqnarray}\label{2.10}
\sup_{0<\varepsilon\leq1}\mathbb{P}\left(\left(|X_n^\varepsilon|_{H,V}(T)\right)^p>M\right)\leq e^{-R/\varepsilon},
\end{eqnarray}
\begin{eqnarray}\label{2.11}
\sup_{0<\varepsilon\leq1}\mathbb{P}\left(\sup_{0\leq t\leq T}\|Y_n^\varepsilon(t)\|_{V}^2>M\right)\leq e^{-R/\varepsilon}.
\end{eqnarray}
For such $M$, by \eref{2.9} and the definition of stoping time $\tau_{M,\varepsilon}^{n}$, there exists $\varepsilon_0$, such that for every $\varepsilon$ satisfying $0<\varepsilon\leq\varepsilon_0$,
\begin{eqnarray}\label{2.12}
&&\mathbb{P}\left(\sup_{0\leq t\leq T}\|X_n^{\varepsilon}(t)-Y_n^{\varepsilon}(t)\|_{H}^2>\delta,\left(|X_n^\varepsilon|_{H,V}(T)\right)^p\leq M,\sup_{0\leq t\leq T}\|Y_n^\varepsilon(t)\|_{V}^2\leq M\right)
\nonumber\\\leq\!\!\!\!\!\!\!\!&&
\mathbb{P}\left(\sup_{0\leq t\leq T\wedge\tau_{M,\varepsilon}^n}\|X_n^{\varepsilon}(t)-Y_n^{\varepsilon}(t)\|_{H}^2>\delta\right)\leq e^{-R/\varepsilon}.
\end{eqnarray}
Combining \eref{2.10}, \eref{2.11} and \eref{2.12}, we conclude that there exists $\varepsilon_0$, such that for every $\varepsilon$ satisfying $0<\varepsilon\leq\varepsilon_0$,
\begin{eqnarray*}
\mathbb{P}\left(\sup_{0\leq t\leq T}\|X_n^{\varepsilon}(t)-Y_n^{\varepsilon}(t)\|_{H}^2>\delta\right)\leq3e^{-R/\varepsilon}.
\end{eqnarray*}
Since $R$ is arbitrary, the assertion of the lemma follows.
\end{proof}
\medskip
We can now complete the proof of our main result.

\medskip
\noindent
\textbf{Proof of Theorem \ref{main result}:}
Due to Lemmas \ref{3.3} and \ref{3.4}, for any $R>0$, there exists a $N_0$ satisfying
\begin{eqnarray}\label{2.13}
\mathbb{P}\left(\sup_{0\leq t\leq T}\|X_{N_0}^{\varepsilon}(t)-X^\varepsilon(t)\|_{H}^2>\frac{\delta}{3}\right)\leq e^{-R/\varepsilon},~~\text{for any}~0<\varepsilon\leq1;
\end{eqnarray}
and
\begin{eqnarray}\label{2.14}
\mathbb{P}\left(\sup_{0\leq t\leq T}\|Y_{N_0}^{\varepsilon}(t)-Y^\varepsilon(t)\|_{H}^2>\frac{\delta}{3}\right)\leq e^{-R/\varepsilon},~~\text{for any}~0<\varepsilon\leq1.
\end{eqnarray}
For such $N_0$, according to Lemma \ref{3.5}, there exists $\varepsilon_0$, such that for every $\varepsilon$ satisfying $0<\varepsilon\leq\varepsilon_0$,
\begin{eqnarray}\label{2.15}
\mathbb{P}\left(\sup_{0\leq t\leq T}\|X_{N_0}^{\varepsilon}(t)-Y_{N_0}^{\varepsilon}(t)\|_{H}^2>\frac{\delta}{3}\right)\leq e^{-R/\varepsilon}.
\end{eqnarray}
Combining \eref{2.13}-\eref{2.15}, for any $0<\varepsilon\leq\varepsilon_0$, we have
\begin{eqnarray*}
\mathbb{P}\left(\sup_{0\leq t\leq T}\|X^{\varepsilon}(t)-Y^{\varepsilon}(t)\|_{H}^2>{\delta}\right)\leq 3e^{-R/\varepsilon}.
\end{eqnarray*}
Since $R$ is arbitrary, we obtain
\begin{eqnarray*}
\lim_{\varepsilon\rightarrow0}\varepsilon\log\mathbb{P}\left(\sup_{0\leq t\leq T}\|X^{\varepsilon}(t)-Y^{\varepsilon}(t)\|_{H}^2>\delta\right)=-\infty,
\end{eqnarray*}
i.e. \eref{Z} holds. Hence the conclusion of Theorem \ref{main result} holds by using the exponential equivalence result of LDP, see e.g. \cite[Theorem 4.2.13]{DZ}.
\hspace{\fill}$\Box$

\section{Application to examples} \label{example}
The main result of this paper is applicable to a large class of SPDE with local monotone coefficients, and we illustrate the applicability of our main result to the following concrete examples
of SPDE models.

In this section we use $\Lambda\subseteq\mathbb{R}^d$ to denote an open bounded domain with a smooth boundary  and
$C_0^\infty(\Lambda, \mathbb{R}^d)$ denote the set of all smooth functions from $\Lambda$ to $\mathbb{R}^d$
with compact support.  For $p\ge 1$, let $\left(L^p(\Lambda, \mathbb{R}^d), \|\cdot\|_{L^p} \right) $ be the vector valued $L^p$-space.
 For any integer $m> 0$, let $W_0^{m,p}(\Lambda, \mathbb{R}^d)$ denote the standard  Sobolev space on $\Lambda$
with values in $\mathbb{R}^d$,  $i.e.$  the closure of $C_0^\infty(\Lambda, \mathbb{R}^d)$ with respect to the following  norm:
$$ \|u\|_{W^{m,p}} = \left( \sum_{0\le |\alpha|\le m} \int_{\Lambda} |D^\alpha u|^pd x \right)^\frac{1}{p}.  $$
For the reader's convenience, we recall the following Gagliardo-Nirenberg interpolation inequality (cf. e.g. \cite[Theorem 2.1.5]{Taira}).

If $m,n\in\mathbb{N}$ and $q\in[1,\infty]$ such that
\[
\frac{1}{q}=\frac{1}{2}+\frac{n}{d}-\frac{m \theta}{d},\ \frac{n}{m}\le\theta\le1,
\]
then there exists a constant $C>0$ such that
\begin{equation}
\|u\|_{W^{n,q}}\le C\|u\|_{W^{m,2}}^{\theta}\|u\|_{L^{2}}^{1-\theta},\ \ u\in W^{m,2}(\Lambda, \mathbb{R}^d).\label{GN_inequality}
\end{equation}

\subsection{Stochastic multidimensional Burgers type equation}
The first example is stochastic multidimensional Burgers type equation.
Consider the Gelfand triple
$$V:=W_0^{1,2}\subset H:=L^2(\Lambda)\subset(W_0^{1,2})^*=V^*$$
and the following
semilinear stochastic partial differential equation
\begin{equation}\label{semilinear}
\left\{ \begin{aligned}
&d X(t)=(\Delta X+\langle f(X),\nabla X\rangle+g(X(t))) d t+ B(X(t))d W(t),\\
 &X(0)=x\in H,
\end{aligned} \right.
\end{equation}
where$f=(f_1,\cdots,f_d):\mathbb{R}\rightarrow\mathbb{R}^d$ is a
Lipschitz functions and $\langle \ , \ \rangle$ denotes the inner
product in $\mathbb{R}^d$, $W$ is a cylindrical Wiener process in $U$ defined
on a probability space $(\Omega,\mathcal {F},\mathcal
{F}_t,\mathbb{P})$. Let
$g:\mathbb{R}\rightarrow \mathbb{R}$ be a continuous function with
$g(0)=0$ such that for some constants $C,r,s\in [0,\infty[$
\begin{equation}\label{5.21}
|g(x)|\leq C(|x|^r+1),  x\in \mathbb{R};
\end{equation}
\begin{equation}\label{5.22}
(g(x)-g(y))(x-y)\leq C(1+|y|^s)(x-y)^2,  x,y\in \mathbb{R}.
\end{equation}

Now, for $\varepsilon>0$, we consider the small time process $X(\varepsilon t)$.
Let $\mu^{\varepsilon}$ be the law of $X(\varepsilon\cdot)$ on $C([0,T],H)$, we have
the small time LDP for \eref{semilinear}.

\begin{thm}\label{SLSPDE} (stochastic multidimensional Burgers type equation) Assume $g$ satisfies the above conditions, $B$ satisfies the assumption \ref{Assumption2}.
If $d=1, r=3, s=2$, or $d=2, r<3, s=2$, or $d=3$, $r=\frac{7}{3}$, $s=\frac{4}{3}$, then (\ref{semilinear})  has a unique solution $X(t)$ and $\mu^{\varepsilon}$ satisfies the LDP with the rate function $I(\cdot)$ given by \eref{Ig}.
\end{thm}
\begin{proof}
According to \cite[Example 5.1.8]{LR15}, we know the coefficients in  \eref{semilinear} satisfies the Hemicontinuity, local monotonicity and growth
properties (A1)-(A3). Therefore, the assertion follows by Lemma \ref{Lem1} and Theorem \ref{main result}.
\end{proof}
\begin{rem}
If $d = 1$, $f(x)= x$ and $g=0$, Theorem \ref{SLSPDE} can be applied to the classical stochastic Burgers equation. Here, we also allow a polynomial
perturbation term $g$ in the drift of \eref{semilinear}. For example, one can take $g(x)=-x^3+c_1 x^2+c_2 x~(c_1 , c_2 \in \mathbb{R})$  and show that
\eref{5.21}-\eref{5.22} hold. Hence \eref{semilinear} also covers some stochastic reaction-diffusion type equations.
\end{rem}
Besides from the example of semilinear SPDE above, we can also apply the main result to the following quasilinear SPDEs such as stochastic $p$-Laplace equation and stochastic porous media equation, which have been studied a lot in recent years see e.g. \cite{Gess1,Gess2,BLR11,Liu09,L,LR15,LR18,MZ19,RZ2,W15}) and references therein.
\subsection{Stochastic $p$-Laplace equation} We consider the triple
$$V:=W_0^{1,p}\subset H:=L^2(\Lambda)\subset(W_0^{1,p})^*=V^*$$ and
the following stochastic $p$-Laplace equation
\begin{equation}\label{PLAP}
\left\{ \begin{aligned}
&dX(t)=[div(|\nabla X(t)|^{p-2}\nabla X(t))-c|X(t)|^{\tilde{p}-2}X(t)]dt+B(X(t))dW (t),\\
 &X(0)=x\in H,
\end{aligned} \right.
\end{equation}
where $2\leq p\leq\infty,1\leq\tilde{p}\leq p$, $c$ is positive constant and $W(t)$ is a cylindrical Wiener process in $U$ defined
on a probability space $(\Omega,\mathcal {F},\mathcal
{F}_t,\mathbb{P})$.

It is well known that the $p$-Laplace operator satisfies the hemicontinuity, monotonicity and growth
properties (A1)-(A3) (see, e.g. \cite[Example 5.5]{L}).

Consider the small time process $X(\varepsilon t)$ and let $\mu^{\varepsilon}$ be the law of $X(\varepsilon\cdot)$ on $C([0,T],H)$, by applying our main result, we formulate the small time LDP for Eq.~\eref{PLAP}.

\begin{thm}(stochastic $p$-Laplace equation)\label{main result PLAP}
Assume that $B$ satisfies the assumption \ref{Assumption2}, then (\ref{PLAP})  has a unique solution $X(t)$ and $\mu^{\varepsilon}$ satisfies the LDP with the rate function $I(\cdot)$ given by \eref{Ig}.
\end{thm}

\subsection{Stochastic porous media equation}
 The main result in this work can also be applied to stochastic porous media equation.
Let $(E,\mathcal{M},\textbf{m})$ be a separable probability space and $(L,\mathcal{D}(L))$ a negative definite self-adjoint linear operator on
$(L^2(\textbf{m}),\langle\cdot,\cdot\rangle)$ ) with spectrum contained in $(-\infty,-\lambda_0]$ for some $\lambda_0>0$. Then the
embedding
$$H^1:=\mathcal{D}(\sqrt{-L})\subseteq L^2(\textbf{m})$$
is dense and continuous. Define $H$ is the dual Hilbert space of $H^1$ realized through
this embedding. Assume $L^{-1}$ is continuous on $L^{r+1}(\textbf{m})$.

For fixed $r > 1$, we consider the following Gelfand triple
$$V:=L^{r+1}(\textbf{m})\subset H:=H\subset V^*$$
and the stochastic porous media equation
\begin{equation}\label{PME}
\left\{ \begin{aligned}
&dX(t)=[L\Psi(t,X(t))+\Phi(t,X(t))]dt+B(X(t))dW (t),\\
 &X(0)=x\in H,
\end{aligned} \right.
\end{equation}
where $\Psi,\Phi:[0,T]\times\mathbb{R}\rightarrow\mathbb{R}$
are measurable and continuous in the second variable, $W(t)$ is a cylindrical Wiener process in $U$ defined
on a probability space $(\Omega,\mathcal {F},\mathcal
{F}_t,\mathbb{P})$. Suppose that there exist two constants $\delta>0$ and $K$ such that
\begin{eqnarray}
&&|\Psi(t,x)|+|\Phi(t,x)|\leq K(1+|x|^r),~~t\in[0,T],x\in \mathbb{R};
\label{PME3}\\&&-\langle\Psi(t,u)-\Psi(t,v),u-v\rangle-\langle\Phi(t,u)-\Phi(t,v),L^{-1}(u-v)\rangle
\nonumber\\\leq\!\!\!\!\!\!\!\!&&-\delta\|u-v\|_V^{r+1}+K\|u-v\|_H^2,~~t\in[0,T],u,v\in V.\label{PME4}
\end{eqnarray}

It is easy to see that the drift part of Eq.~\eref{PME} satisfies the conditions (A1)-(A3) (cf. \cite[Example 5.3]{L}).
Let $\mu^{\varepsilon}$ be the law of $X(\varepsilon\cdot)$ on $C([0,T],H)$, by applying our main result, we formulate the small time LDP for Eq.~\eref{PME}.

\begin{thm}(stochastic porous media equation)\label{main result PME2}
Assume that $\Psi,\Phi$ satisfy the above conditions \eref{PME3}-\eref{PME4} and $B$ satisfies the assumption \ref{Assumption2}, then (\ref{PME})  has a unique solution $X(t)$ and  $\mu^{\varepsilon}$ satisfies the LDP with the rate function $I(\cdot)$ given by \eref{Ig}.
\end{thm}
\begin{rem}
If we take $L=\Delta$, the Laplace operator on a smooth bounded domain in a
complete Riemannian manifold with Dirichlet boundary condition. A simple example
for $\Psi$ and $\Psi$ satisfy the above conditions \eref{PME3}-\eref{PME4} is given by
$$\Psi(t,x)= f(t)|x|^{r-1}x, ~~~~\Phi(t,x)= g(t)x$$
for some strictly positive continuous function $f$ and bounded function $g$ on [0,T].
\end{rem}

In the following, we will show that the main result is also applicable to many stochastic hydrodynamical systems.
\subsection{Stochastic 2D Navier-Stokes equation}
Our next example is the stochastic 2D Navier-Stokes equation.
The classical Navier-Stokes equation is a very important model in fluid mechanics to describe the time evolution of
incompressible fluids, it can be formulated as follows (2D case):
\begin{equation}\label{NS}
\partial_tu=\nu\Delta u-(u\cdot\nabla)u-\nabla p+f,~~div(u)=0,
\end{equation}
where $u=(u_1(x,t),u_2(x,t))$ is the velocity of a fluid, $p$ is the pressure, $\nu$
is the Kinematic viscosity, and $f$ denote the
external force of the fluid, and
$$u\cdot\nabla=\sum^d_{j=1}u_j\partial_j.$$

 Let $\Lambda\subset\mathbb{R}^2$ be an open bounded domain with smooth
boundary. Define
$$V:=\{v\in W_0^{1,2}(\Lambda,\mathbb{R}^2):div(v)=0\},~~\|v\|_V:=\left(\int_{\Lambda}|\nabla v|^2dx\right)^{1/2},$$
and $H$ is the closure of $V$ in the following norm
$$\|v\|_H:=\left(\int_{\Lambda}| v|^2dx\right)^{1/2}.$$
We define the stokes operator $A$ by
$$Au= P_H\Delta u,~~\forall u\in W^{2,2}(\Lambda,\mathbb{R}^2)\cap V,$$
where $P_H$ (Helmholtz-Leray projection) is the projection operator from $ L^2(\Lambda,\mathbb{R}^2)$ to $H$, and the nonlinear operator
$$F(u,v)=-P_H((u\cdot\nabla)v)),~~F(u)=F(u,u).$$
Then \eref{NS} can then be written in  form:
$$\partial_tu=\nu Au+F(u)+f(x),~u(0)=u_0.$$
Now, we study the following stochastic 2D Navier-Stokes equation
\begin{equation}\label{NS1}
\left\{ \begin{aligned}
&dX(t)=(\nu AX(t)+F(X(t))+f(x))dt+B(X(t))dW(t),\\
 &X(0)=x\in H,
\end{aligned} \right.
\end{equation}
where $W(t)$ is a cylindrical Wiener process in $U$ defined
on a probability space $(\Omega,\mathcal {F},\mathcal
{F}_t,\mathbb{P})$.

It is well known that stochastic 2D Navier-Stokes equation satisfies the conditions (A1)-(A3) (see, e.g. \cite[Example 3.3]{LR10}).

Let $\mu^{\varepsilon}$ be the law of $X(\varepsilon\cdot)$ on $C([0,T],H)$. By applying our main result, we have
the small time LDP for stochastic 2D Navier-Stokes equation \eref{NS1}.
\begin{thm}(stochastic 2D Navier-Stokes equation)\label{2DNS}
Assume that $B$ satisfies the assumption \ref{Assumption2}, then (\ref{NS1})  has a unique solution $X(t)$ and  $\mu^{\varepsilon}$ satisfies the LDP with the rate function $I(\cdot)$ given by \eref{Ig}.
\end{thm}

\begin{rem}
(1) The small time LDP for stochastic 2D Navier-Stokes equation have been established by Xu and Zhang \cite{XZ09}.

(2) Beside the stochastic 2D Navier-Stokes equation, many other hydrodynamical systems also satisfy the local monotonicity condition (A2) and growth condition (A3). For example, Chueshov and Millet \cite{CM10} have studied the well-posedness and small noise LDP for an abstract stochastic evolution equations, covering a wide class of fluid dynamical models such
as stochastic 2D Boussinesq equations, stochastic 2D magneto-hydrodynamic equations, stochastic 2D magnetic B\'{e}nard problem, stochastic 3D Leray-$\alpha$ model and also shell models of turbulence.
We refer the reader to \cite{CM10} (and the references therein) for the details of these models. Note that the assumptions in \cite{CM10} imply the conditions (A1)-(A3) (cf. \cite[section 3.1]{MZ19} for a detail proof).

(3) Furthermore, below we will show that the main result in this work is also applicable to
stochastic power law fluid equation and stochastic Ladyzhenskaya model.
\end{rem}

\subsection{Stochastic power law fluid equation}
As one of the important models in hydrodynamical, stochastic power
law fluid equation can be used to characterize the dynamic
properties of various incompressible non-Newtonian fluids. We can
refer to \cite{FR08,MNRR} for the study of this type of equation.

Let $\Lambda$ be the open bounded domain with smooth boundary on
$\mathbb{R}^d$($d\geq2$), $u:\Lambda\rightarrow \mathbb{R}^d$ be a
vector field. Define
$$e(u):\Lambda\rightarrow \mathbb{R}^d\otimes
\mathbb{R}^d;
~~e_{i,j}(u)=\frac{\partial_{i}u_{j}+\partial_{j}u_{i}}{2},i,j=1,\cdots,d.$$
$$\tau(u):\Lambda\rightarrow \mathbb{R}^d\otimes
\mathbb{R}^d; ~~~\tau(u)=2\nu(1+|e(u)|)^{p-2}e(u),$$ where $\nu>0$
is the viscosity coefficient of the fluid, $p>1$ is a constant.

Now we study a hydrodynamic equation with a power law property:
$$\partial_{t}u=div(\tau(u))-(u\cdot\nabla)u-\nabla
p+f,div(u)=0,$$ where $u=u(t,x)=(u_i(t,x))^d_{i=1}$ denote the
velocity field of the fluid, $p$ is pressure, $f$ denote the
external force of the fluid,
$$div(\tau(u))=\left(\sum^d_{j=1}\partial_j\tau_{i,j}(u)\right)^d_{i=1}.$$
The power law fluid equation defined above is  the classical
Navier-Stokes equation if $p=2$.

Now we consider the Gelfand triple:
$$V\subseteq H\subseteq V^*,$$ where $$
\setlength{\abovedisplayskip}{8pt}
\setlength{\belowdisplayskip}{8pt} V=\{u\in W^{1,p}_0(\Lambda;
\mathbb{R}^d): div(u)=0\}; \ H=\{u\in L^2(\Lambda;\mathbb{R}^d):
div(u)=0\}.$$ Let $P_H$ be the projection operator on
$L^2(\Lambda;\mathbb{R}^d)\rightarrow H$. Then we can extend the
operator
$$\mathcal{A}: W^{2,p}(\Lambda;\mathbb{R}^d)\cap V\rightarrow
H,~~\mathcal{A}(u)=P_H[div(\tau(u))];$$
$$F:\left( W^{2,p}(\Lambda;\mathbb{R}^d)\cap V  \right)  \times
\left( W^{2,p}(\Lambda;\mathbb{R}^d) \cap V \right) \rightarrow H;$$
$$F(u,v)=-P_H[(u\cdot\nabla)v], F(u):=F(u,u)$$
to the map(see \cite{LR13}):
$$\mathcal{A}:V\rightarrow V^*;  ~~F:V\times
V\rightarrow V^*.$$ In particular, we have
$$\langle \mathcal{A}(u),v\rangle_{V}=-\int_\Lambda
\sum_{i,j=1}^{d}\tau_{i,j}(u)e_{i,j}(v) dx,~u,v\in V;$$
$${ }_{V^*}\langle F(u,v),w\rangle_V=-{ }_{V^*}\langle F(u,w),v\rangle_V,  ~~{ }_{V^*}\langle
F(u,v),v\rangle_V=0,~u,v,w\in V.$$ Then the power law fluid equation
defined above can be written in variational form:
$$\partial_{t}u=\mathcal{A}u(t)+F(u(t))+f(t),~~u(0)=u_0.$$
Now study the following stochastic power law fluid equation
\begin{equation}\label{PLF}
\left\{ \begin{aligned}
&dX(t)=(\nu \mathcal{A}X(t)+F(X(t))+f(x))dt+B(X(t))dW(t),\\
 &X(0)=x\in H,
\end{aligned} \right.
\end{equation}
where $W(t)$ is a cylindrical Wiener process in $U$ defined
on a probability space $(\Omega,\mathcal {F},\mathcal
{F}_t,\mathbb{P})$.

Let $\mu^{\varepsilon}$ be the law of $X(\varepsilon\cdot)$ on $C([0,T],H)$.
We  will show the small time LDP and its proof by applying our main result.

\begin{thm} (stochastic power law fluid equation)\label{SPLF} Let $p\geq\frac{d+2}{2}$ and $B$ satisfy the assumption \ref{Assumption2}. Then (\ref{PLF})  has a unique solution $X(t)$ and  $\mu^{\varepsilon}$ satisfies the LDP with the rate function $I(\cdot)$ given by \eref{Ig}.
\end{thm}
\begin{proof} Assume without loss of generality that viscosity coefficient $\nu=1$.
By \cite[Lemma 1.19]{MNRR}, we have
$$\int_\Lambda|e(u)|^p dx\geq
C_{p}\|u\|_{W^{1,p}},u\in W^{1,p}_0(\Lambda;\mathbb{R}^d);$$
$$\sum^d_{i,j=1}\tau_{i,j}(u)e_{i,j}(u)\geq C(|e(u)|^p-1);$$
$$\sum_{i,j=1}^
{d}(\tau_{i,j}(u)-\tau_{i,j}(v))(e_{i,j}(u)-e_{i,j}(v))\geq
C(|e(u)-e(v)|^2+|e(u)-e(v)|^p);$$
$$|\tau_{i,j}(u)|\leq C(1+|e(u)|)^{p-1},i,j=1...,d.$$
According to the inequality above, for any $u,v\in V$, we have
\begin{eqnarray*}
  { }_{V^*}\langle F(u)-F(v),u-v\rangle_V =&&\!\!\!\!\!\!\!\! -{ }_{V^*}\langle F(u-v),v\rangle_V \\
 =&&\!\!\!\!\!\!\!\!{ }_{V^*}\langle F(u-v,v),u-v\rangle_V\\
   \leq&&\!\!\!\!\!\!\!\! C\|v\|_V\|u-v\|_{L^{\frac{2p}{p-1}}}^{2}\\
  \leq&&\!\!\!\!\!\!\!\!
  C\|v\|_V\|u-v\|_{W^{1,2}}^{\frac{d}{p}}\|u-v\|_{H}^{\frac{2p-d}{p}}\\
   \leq&&\!\!\!\!\!\!\!\! \varepsilon\|u-v\|_{W^{1,2}}^{2}+C_\varepsilon\|v\|_V^{\frac{2p}{2p-d}}\|u-v\|_H^2.
\end{eqnarray*}
Then
\begin{eqnarray*}
&&{ }_{V^*}\langle
\mathcal{A}(u)+F(u)-\mathcal{A}(v)-F(v),u-v\rangle_V \\
=\!\!\!\!\!\!\!\!&&-\int_\Lambda\Sigma_{i,j=1}^{d}(\tau_{i,j}(u)-\tau_{i,j}(v))(e_{i,j}(u)-e_{i,j}(v))\
d x \\
\leq\!\!\!\!\!\!\!\!&&-C\|e(u)-e(v)\|_H^2\\
\leq\!\!\!\!\!\!\!\!&&-C\|u-v\|_{W^{1,2}}^{2}.
\end{eqnarray*}
Thus we have
$${ }_{V^*}\langle \mathcal{A}(u)+F(u)-\mathcal{A}(v)-F(v)\rangle_V
\leq -(C-\varepsilon)\|u-v\|_{W^{1,2}}^{2}
+C_{\varepsilon}\|v\|_V^{\frac{2p}{2p-d}}\|u-v\|_H^2,$$ the condition (A2) holds with $\rho(v) = C_{\varepsilon} \|v \|_{V}^{\frac{4q}{4q-d}}$ and $\alpha = p$.

Note that $$|{ }_{V^*}\langle F(v),u\rangle_V|=|{ }_{V^*}\langle
F(v,u),v\rangle_V|\leq\|u\|_V\|v\|_{L^{\frac{2p}{p-1}}}^{2},u,v\in V,$$
Then we have $$\|F(v)\|_{V^*}\leq\|v\|_{L^{\frac{2p}{p-1}}}^2,v\in V.$$
Let $q=\frac{dp}{d-p}, \gamma=\frac{d}{(d+2)p-2d}$, by the Gagliardo-Nirenberg interpolation inequality \eref{GN_inequality} we have
$$\|v\|_{L^{\frac{2p}{p-1}}}\leq\|v\|_{L^q}^{\gamma}\|v\|_{L^2}^{1-\gamma}\leq
C\|v\|_V^\gamma\|v\|_H^{1-\gamma}.$$ Since $p\geq\frac{d+2}{2}$, the
condition (A3) holds.

Therefore, the assertion follows by Lemma \ref{Lem1} and Theorem \ref{main result}.
\end{proof}
\subsection{Stochastic Ladyzhenskaya model}\label{sub:Ladyzhenskaya}
The Ladyzhenskaya model is a higher order variant of the power law fluid where the stress tensor has the form
$$
	\tilde{\tau}(u) \colon \Lambda \rightarrow \mathbb{R}^{d} \otimes \mathbb{R}^{d}, \quad \tilde{\tau}(u) = 2 \mu_{0}(1 + |e(u)|^{2})^{\frac{p-2}{2}} e(u) - 2 \mu_{1} \Delta e(u).
$$
This model was pioneered by Ladyzhenskaya \cite{L70} and further analyzed by various authors (see \cite{ZD10} and the references therein). Compared to the power law fluids considered above, there is an additional fourth order term $div(- 2 \mu_{1} \Delta e(u) )$ present in the equation. The fluids are shear thinning when $1 < p < 2$ and
shear thickening when $p > 2$.

Martingale and stationary solutions for this model was established by Guo et al. in \cite{GGZ10}.  Moreover, the existence of random attractors for this model has been proved for $p \in (1,2)$, i.e. shear-thinning fluids, by Duan and Zhao in \cite{ZD10}.  Recently, the small time LDP for this model has been studied for $d=2,p \in (1,\frac{5}{2}]$ by Lin and Sun in \cite{LS17}.

Consider the Gelfand triple $V \subset H \subset V^{*}$, where
\begin{align*}
V & =\left\{ u\in W_{0}^{2,2}(\Lambda;\mathbb{R}^{d}):\ div(u)=0\ \text{in}\ \Lambda\right\} ;\\
H & =\left\{ u\in L^{2}(\Lambda;\mathbb{R}^{d}):\ div(u)=0\ \text{in}\ \Lambda,\ u\cdot n=0\ \text{on}\ \partial\Lambda\right\} .
\end{align*}

Let $P_{H}$ be the orthogonal (Helmholtz-Leray) projection from $L^{2}(\Lambda,\mathbb{R}^{d})$ to $H$. Similar to Theorems
\ref{2DNS} and \ref{SPLF}, the operators
\begin{align*}
& \mathcal{\tilde{A}}: C^{\infty}_{c}(\Lambda;\mathbb{R}^{d})\cap V\rightarrow H,\ \mathcal{N}(u):=P_{H}\left[\text{div}(\tilde{\tau}(u))\right]; \\
&F: \left( C^{\infty}_{c}(\Lambda;\mathbb{R}^{d})\cap V \right) \times \left( C^{\infty}_{c}(\Lambda;\mathbb{R}^{d})\cap V \right) \rightarrow H; \\
&F(u,v):=-P_{H}\left[(u\cdot\nabla)v\right],\ F(u):=F(u,u);
\end{align*}
can be extended to the well defined operators:
\[
\mathcal{\tilde{A}}:V\rightarrow V^{*};\  F:V\times V\rightarrow V^{*}.
\]

With these preparations, we can write our model in the abstract form
\begin{equation}\label{Ladyzhenskaya}
\left\{ \begin{aligned}
&dX(t)=(\mathcal{\tilde{A}}(X(t))+F(X(t))+f(x))dt+B(X(t))dW(t),\\
 &X(0)=x\in H,
\end{aligned} \right.
\end{equation}
where $W(t)$ is a cylindrical Wiener process in $U$ defined
on a probability space $(\Omega,\mathcal {F},\mathcal
{F}_t,\mathbb{P})$.

Let $\mu^{\varepsilon}$ be the law of $X(\varepsilon\cdot)$ on $C([0,T],H)$.
We then have the small time LDP by applying our main result, which covers the result in \cite{LS17}. Since the proof is similar with power law fluid in Theorem \ref{SPLF}, we omit it here, the reader might refer to \cite{LS17} for some further detailed calculations.

\begin{thm} (stochastic Ladyzhenskaya model)\label{Ladyzhenskaya1} Let $d=2,p \in (1,\frac{5}{2}]$ and $B$ satisfy the assumption \ref{Assumption2}. Then Eq.~(\ref{Ladyzhenskaya}) has a unique solution $X(t)$ and  $\mu^{\varepsilon}$ satisfies the LDP with the rate function $I(\cdot)$ given by \eref{Ig}.
\end{thm}

\begin{rem} The restriction on parameter $p$ allows us to understand the nonlinear term as a perturbation of the linear term. In fact, for general $d\geq2$, using the Gagliardo-Nirenberg inequality, it is possible to find a ``maximal" range $(1,p_{d}] $ of $p$ to which the locally monotone variational framework can apply (cf. \cite{GLS}).
\end{rem}

\vspace{0.3cm}
\noindent\textbf{Acknowledgment} This work is supported by NSFC (No. 11571147,11771187,11822106,
11831014), NSF
of Jiangsu Province
 (No. BK20160004), the
PAPD of Jiangsu Higher Education Institutions.


\begin{thebibliography}{}
\bibitem{A13}  H. Abdallah,  \emph{A Varadhan type estimate on manifolds with time-dependent metrics and constant volume}, J. Math. Pures Appl. \textbf{99}  (2013), 409--418.

\bibitem{AK98}  S. Aida, H. Kawabi, \emph{Short time asymptotics of a certain infinite dimensional diffusion process}, In Stochastic Analysis and Related Topics,
VII (Kusadasi, 1998) 77--124. Progr. Probab. \textbf{48}. Birkh\"{a}user Boston, Boston, MA, 2001.

\bibitem{AZ02} S. Aida, T. S. Zhang, \emph{On the small time asymptotics of diffusion processes on path groups}, Potential Anal. \textbf{16} (2002), 67--78.

\bibitem{AH05} T. Ariyoshi, M. Hino, \emph{Small-time asymptotic estimates in local Dirichlet spaces}, Electron. J. Probab. \textbf{10} (2005), 1236--1259.

\bibitem{ABBF}  M. Avellaneda, D. Boyer-Olson, J. Busca, P. Friz, \emph{Application of large deviation methods to the pricing of index
options in finance}, C. R. Math. Acad. Sci. Paris \textbf{336} (2003), 263--266.


\bibitem{BBF04} H. Berestycki, J. Busca, I. Florent, \emph{Computing the implied volatility in stochastic volatility models}, Comm. Pure Appl. Math. \textbf{57} (2004), 1352--1373.

\bibitem{Br97} Z. Brz\'{e}zniak, \emph{On stochastic convolution in Banach spaces and applications}, Stochastics
Stochastics Rep. \textbf{61} (1997), 245--295.

\bibitem{BP99} Z. Brz\'{e}zniak, S. Peszat, \emph{Space-time continuous solutions to SPDEs driven by a homogeneous Wiener process}, Studia Math. \textbf{137} (1999), 261--299.
\bibitem{BD00} A. Budhiraja, P. Dupuis, \emph{A variational
 representation for positive functionals
of infinite dimensional Brownian motion,}  Probab. Math. Statist.
 \textbf{20} (2000), 39--61.
\bibitem{BDM} A. Budhiraja, P. Dupuis, V. Maroulas, \emph{Large deviations for infinite
dimensional stochastic dynamical systems,}  Ann. Probab. \textbf{36} (2008),
1390--1420.





\bibitem{CG} Y. Chen, H. Gao, L. Fan, \emph{Well-posedness and the small time large deviations of the stochastic integrable equation governing short-waves in a long-wave model}, Nonlinear Anal. Real World Appl.  \textbf{29}  (2016), 38--57.

\bibitem{CFZ} Z.-Q. Chen, S. Fang, T.S. Zhang, \emph{Small time asymptotics for Brownian motion with singular drift}, Proc. Amer. Math. Soc. (2019), https://doi.org/10.1090/proc/14511.

\bibitem{C} P.L. Chow, \emph{Large deviation problem for some parabolic
 It\^{o} equations,} Comm. Pure Appl. Math. \textbf{45} (1992), 97--120.

\bibitem{CM10} I. Chueshov, A. Millet, \emph{Stochastic 2D Hydrodynamical Type Systems: Well Posedness and Large Deviations,}
 Appl. Math. Optim. \textbf{61} (2010), 379--420.

\bibitem{CDF} E.A. Coayla-Teran, P.M. Dias de Magalh$\tilde{a}$es, J. Ferreira,
\emph{Existence of optimal controls for SPDE with locally monotone coefficients}, International J. Control (2018),  https://doi.org/10.1080/00207179. 2018.1508849.





\bibitem{Da} B. Davis, \emph{On the $L^p$-norms of stochastic integrals and other martingales}, Duke Math. J. \textbf{43} (1976), 697--704.

\bibitem{DZ} A. Dembo, O. Zeitouni, \emph{Large deviations
techniques and applications,}  Jones and Bartlett, Boston, 1993.



\bibitem{DZ18} Z. Dong, R. Zhang, \emph{On the small-time asymptotics of 3D stochastic primitive equations}, Math. Methods Appl. Sci.  \textbf{41}  (2018), 6336--6357.




\bibitem{FZ99} S. Fang, T.S. Zhang, \emph{On the small time behavior of Ornstein-Uhlenbeck processes with unbounded linear drifts}, Probab. Theory
Related Fields \textbf{114} (1999), 487--504.

\bibitem{FFK12}J. Feng, J.-P. Fouque, R. Kumar, \emph{Small-time asymptotics for fast mean-reverting stochastic volatility models}, Ann. Appl. Probab.\textbf{22} (2012), 1541--1575.



\bibitem{FJ11} M. Forde, A. Jacquier, \emph{Small-time asymptotics for an uncorrelated local-stochastic volatility model}, Appl. Math. Finance, \textbf{18} (2011), 517--535.

\bibitem{FJ12} M. Forde, A. Jacquier, R. Lee, \emph{The small-time smile and term structure of implied volatility under the Heston model}, SIAM J. Financial Math. \textbf{3} (2012), 690--708.
\bibitem{FR08}  J. Frehse, M. Ru\v{z}i\v{c}ka, \emph{Non-homogeneous generalized Newtonian fluids}, Math. Z. \textbf{260} (2008), 355--375.

\bibitem{FW} M.I. Freidlin, A.D. Wentzell, \emph{Random
 perturbations of dynamical systems,} Translated from the Russian by
 Joseph Sz\"{u}cs.
 Grundlehren der Mathematischen Wissenschaften [Fundamental Principles
 of Mathematical Sciences],  Springer-Verlag, New York, 260, 1984.

\bibitem{Gess1}B. Gess, \emph{Random attractors for singular stochastic evolution equations},
J. Differential Equations  \textbf{255} (2013), 524--559.

\bibitem{Gess2}B. Gess, \emph{Random attractors for degenerate stochastic partial differential equations}, J. Dynam. Differential Equations \textbf{25} (2013), 121--157.


\bibitem{BLR11}B. Gess, W Liu,  M. R\"{o}ckner, \emph{Random attractors for a class of stochastic partial differential equations driven by general additive noise},
J. Differential Equations  \textbf{251}  (2011), 1225--1253.


\bibitem{GLS}B. Gess, W Liu,  A. Schenke, \emph{Random attractors for locally monotone stochastic partial differential equations},
Preprint.

\bibitem{GGZ10}B. Guo, C. Guo, J. Zhang, \emph{Martingale and stationary solutions for stochastic non-Newtonian fluids,}
Differential Integral Equations \textbf{23} (2010), 303--326.


\bibitem{HM}M. Hino, K. Matsuura, \emph{An integrated version of Varadhan's asymptotics for
lower-order perturbations of strong local Dirichlet forms}, Potential Anal. \textbf{48} (2018),
257--300.

\bibitem{HR03} M. Hino, J. Ramirez, \emph{Small-time Gaussian behaviour of symmetric diffusion semigroup}, Ann. Probab. \textbf{31} (2003), 1254--1295.

\bibitem{J11} T. Jegaraj, \emph{Small time asymptotics for stochastic evolution equations}, J. Theoret. Probab.  \textbf{24} (2011), 756--788.


\bibitem{KR} N.V. Krylov, B.L. Rozovskii, \emph{Stochastic evolution
 equations,}
Translated from Itogi Naukii Tekhniki, Seriya Sovremennye Problemy
Matematiki \textbf{14} (1979), 71--146, Plenum Publishing Corp. 1981.





\bibitem{L70} O.A. Ladyzhenskaya, \emph{New equations for the description of the viscous incompressible fluids
and solvability in large of the boundary value problems for them}, volume V of Boundary Value
Problems of Mathematical Physics. 1970.


\bibitem{LS17} H. Liu, C. Sun, \emph{On the small time asymptotics of stochastic non-Newtonian fluids}, Math. Methods Appl. Sci.  \textbf{40}  (2017), 1139--1152.

 \bibitem{Liu09}  W. Liu, \emph{Harnack inequality and applications for stochastic evolution equations with monotone drifts}, J. Evol. Equ. \textbf{9} (2009), 747--770.


 \bibitem{L} W. Liu, \emph{Large deviations for stochastic evolution equations with small multiplicative noise},
  Appl. Math. Optim. \textbf{61} (2010), 27--56.


\bibitem{LR10} W. Liu, M. R\"{o}ckner, \emph{SPDE in Hilbert Space with Locally Monotone Coefficients}, J. Funct. Anal. \textbf{259} (2010), 2902--2922.

\bibitem{LR13} W. Liu, M. R\"{o}ckner, \emph{Local and global well-posedness of SPDE with generalized coercivity conditions},
 J. Differential Equations \textbf{254} (2013), 725--755.

\bibitem{LR15} W. Liu, M. R\"ockner, \emph{Stochastic Partial Differential Equations: An Introduction,}
 Universitext, Springer, 2015.

\bibitem{LR18} W. Liu, M. R\"ockner, J. L. da Silva, \emph{Quasi-Linear (Stochastic) Partial Differential Equations with Time-Fractional Derivatives},  SIAM J. Math. Anal.  \textbf{50 }(2018), 2588--2607.

\bibitem{LRZ}
W. Liu, M. R\"{o}ckner, X.-C. Zhu, \emph{Large deviation
principles for the stochastic quasi-geostrophic equations,}
  Stochastic Process. Appl. \textbf{123} (2013), 3299--3327.

\bibitem{LTZ} W. Liu, C. Tao, J. Zhu, \emph{Large Deviation Principle for a Class of SPDE with Locally Monotone Coefficients}, Sci. China Math. (2019),   https://doi.org/10.1007/s11425-018-9440-3.

\bibitem{MZ19} T. Ma, R.-C. Zhu, \emph{Wong-Zakai approximation and support theorem for SPDEs with locally monotone coefficients}, J. Math. Anal. Appl. \textbf{469 } (2019), 623--660.

\bibitem{MNRR} J. M\'{a}lek, J. Ne\v{c}as, M. Rokyta, M. Ru\v{z}i\v{c}ka, \emph{Weak andbmeasure-valued solutions to evolutionary PDEs,}  Chapman \& Hall, London, 13, 1996.
\bibitem{Pa} E. Pardoux, \emph{Equations aux d\'eriv\'ees partielles
stochastiques non lin\'eaires monotones,}
Thesis, Universit\'e  Paris XI, 1975.







\bibitem{RZ2} J. Ren, X. Zhang, \emph{Freidlin-Wentzell's Large
 Deviations for Stochastic Evolution Equations,} J. Funct. Anal.  \textbf{254} (2008),
3148--3172.




		
\bibitem{RZ12} M. R\"{o}ckner,  T.S. Zhang, \emph{Stochastic 3D tamed Navier-Stokes equations: existence, uniqueness and small time large deviation principles},
 J. Differential Equations  \textbf{252}  (2012), 716--744.

\bibitem{Se10} J. Seidler, \emph{Exponential estimates for stochastic convolutions in 2-smooth Banach spaces}, Electron. J. Probab. \textbf{15}  (2010), 1556--1573.





\bibitem{St} D.W. Stroock, \emph{An Introduction to the Theory of Large
 Deviations,} Spring-Verlag, New York, 1984.

\bibitem{Taira} K. Taira, \emph{Analytic semigroups and semilinear initial boundary value problems}, Cambridge University Press, 1995.


\bibitem{Tr} H. Triebel, \emph{Interpolation Theory, Function Spaces, Differential Operators}, 2nd ed., Johann
Ambrosius Barth, Heidelberg, 1995.


\bibitem{Va1} S.R.S. Varadhan, \emph{Asymptotic probabilities and
 differential equations,}
  Comm. Pure Appl. Math.  \textbf{19} (1966), 261--286.

\bibitem{Va3} S.R.S. Varadhan, \emph{On the behavior of the fundamental solution of the heat equation with variable coefficients}, Comm. Pure Appl. Math. \textbf{20} (1967), 431-455.

\bibitem{Va2} S.R.S. Varadhan, \emph{Diffusion processes in a small
time interval,} Comm. Pure Appl. Math. \textbf{20}  (1967), 659--685.


\bibitem{W15} F-Y. Wang, \emph{Exponential convergence of non-linear monotone SPDEs},
Discrete Contin. Dyn. Syst. \textbf{35 } (2015), 5239--5253.




\bibitem{XZ18} J. Xiong, J. Zhai, \emph{Large deviations for locally monotone stochastic partial differential equations driven by L\'{e}vy noise},  Bernoulli  \textbf{24} (2018),  2842--2874.

\bibitem{XZ09} T. Xu, T.S. Zhang, \emph{On the small time asymptotics of the two-dimensional stochastic Navier-Stokes equations},  Ann. Inst. Henri Poincar\'{e} Probab. Stat.  \textbf{45}  (2009),  1002--1019.

\bibitem{Zhang19}  R. Zhang, \emph{On the small time asymptotics of scalar stochastic conservation laws}, arXiv:1907.03397.

\bibitem{Zh00}  T.S. Zhang, \emph{On the small time asymptotics of diffusion processes on Hilbert spaces}, Ann. Probab. \textbf{28} (2000), 537--557.

\bibitem{Zh10} X. Zhang, \emph{Stochastic Volterra equations in Banach spaces and stochastic partial differential equation},
J. Funct. Anal. \textbf{258} (2010), 1361--1425.
\bibitem{ZD10} C. Zhao, J. Duan, \emph{Random attractor for the Ladyzhenskaya model with additive
noise}, J. Math. Anal. Appl. \textbf{362} (2010), 241--251.
\bibitem{ZBL} J. Zhu, Z. Brzezniak, W. Liu, \emph{Maximal inequalities and exponential estimates for stochastic
convolutions driven by L\'{e}vy -type processes in Banach spaces with application to stochastic quasi-geostrophic equations}, SIAM J. Math. Anal.  \textbf{51} (2019), 2121--2167.
\end{thebibliography}
\end{document}